\documentclass[12pt]{amsart}
\usepackage{geometry} 
\usepackage{amscd}
\usepackage{amssymb}
\usepackage{amsmath}
\usepackage{amsthm}

\newtheorem{thm}{Theorem}[section]
\newtheorem{lemma}[thm]{Lemma}
\newtheorem{corollary}[thm]{Corollary}
\newtheorem{prop}[thm]{Proposition}

\theoremstyle{definition}

\newtheorem{rem}[thm]{Remark}

\newtheorem{defn}[thm]{Definition}

\makeatletter

\newcommand{\isom}{\overset{\sim}{\rightarrow}}

\def\ge{{\geqq}}

\title{Deformations of trianguline $B$-pairs.}
\author{Kentaro Nakamura}
\date{} 

\begin{document}

\maketitle
\pagestyle{plain}
\footnote{2008 Mathematical Subject Classification 11F80 (primary), 11F85, 11S25 (secondary).
Keywords: $p$-adic Hodge theory, trianguline representations, B-pairs.}
\begin{abstract}
The aim of this article is to study deformation theory of 
trianguline $B$-pairs for any $p$-adic field. For benign $B$-pairs, a special 
good class of trianguline $B$-pairs, we prove a main theorem concerning 
tangent spaces of these deformation spaces. These are generalizations 
of Bella\"iche-Chenevier's and Chenevier's works in the case of $K=\mathbb{Q}_p$,
 where they used ($\varphi,\Gamma$)-modules over Robba ring instead of using $B$-pairs. 
 The main theorem, the author hopes, will play crucial roles in some problems of Zariski density 
 of modular points or of crystalline points in some deformation spaces 
of global or local $p$-adic  Galois representations.
\end{abstract}
\setcounter{tocdepth}{1}
\tableofcontents

\section{Introduction.}
\subsection{Background.}Let $p$ be a prime number and $K$ be a $p$-adic field, i.e. finite 
extension of $\mathbb{Q}_p$. Trianguline representations, which form a class of $p$-adic representations 
of $G_K:=\mathrm{Gal}(\bar{K}/K)$, play some crucial roles in the study of $p$-adic family of global or local 
Galois representations. The aim of this article is, for any $p$-adic field $K$, to develop the deformation 
theory of trianguline representations by using $B$-pairs.

In the study of $p$-adic local Langlands correspondence for $\mathrm{GL}_2(\mathbb{Q}_p)$, in [Co08], Colmez 
defined the notion of trianguline representations, which is a class of $p$-adic representations of $G_K:=\mathrm{Gal}(\bar{K}/K)$ 
and is defined by using Fontaine, Kedlaya's theory of ($\varphi,\Gamma$)-modules over the Robba ring.
Since his work, when $K=\mathbb{Q}_p$, trianguline representations and their deformation theoretic properties have turned out to be 
very important in number theory, especially in the study of $p$-adic local Langlands correspondence by Colmez ([Co09]), or in the study of $p$-adic 
families of $p$-adic modular forms (more precisely, in the study of eigenvarieties), for example, by Kisin ([Ki03]), by Bella\"iche-Chenevier ([Bel-Ch]) and by Chenevier ([Ch08],[Ch09]). 

For any $p$-adic field $K\not= \mathbb{Q}_p$ case, in [Na09], the author of this article generalized 
many results in [Co08] for any $p$-adic fields. Namely,  he studied some fundamental properties of trianguline representations and then classified two dimensional trianguline representations. In [Na09], we studied trianguline representations by using 
$B$-pairs, which was defined by Berger in [Be08], instead of using ($\varphi,\Gamma$)-modules over the Robba ring.
The aim of this article is, by using the results in [Na09] and by using the theory of $B$-pairs, to generalize many results in [Bel-Ch], [Ch08], [Ch09] concerning to 
deformations of trianguline representations to any $p$-adic fields. We develop the deformation theory of 
trianguline representations (more generally, trianguline $B$-pairs) and generalize Chenevier's theorem in [Ch09] 
concerning tangent spaces of these deformation rings to any $p$-adic fields. 

When $K=\mathbb{Q}_p$, these results are crucial 
for $p$-adic Hodge theoretic study of eigenvarieties, especially for problems about Zariski density of modular Galois representations 
(or crystalline representations) in deformation spaces of global (or local) $p$-adic representations ([Ch08].[Ch09],[Co08],[Ki08b]).
In the next article, by using the results of this article (and by generalizing some results in [Ki03] to any $p$-adic fields), 
the author will prove Zariski density of crystalline representations in deformation spaces of two dimensional 
$p$-adic representations for any $p$-adic fields.

  \subsection{Overview.}
 
 We explain details of this article.
 
 In $\S$ 2, we recall the definition of $B$-pairs 
 and study some fundamental properties of trianguline 
 $B$-pairs.
 Let $E$ be a suitable finite extension of $\mathbb{Q}_p$ as in Notation below. In $\S$ 2.1, we recall the definition of $E$-$B$-pairs of $G_K$, which is the $E$-coefficient version of $B$-pairs. We write $B_e:=B_{\mathrm{cris}}^{\varphi=1}$. 
 An $E$-$B$-pair is a pair $W=(W_e,W^+_{\mathrm{dR}})$ where $W_e$ is a finite free $B_e\otimes_{\mathbb{Q}_p}E$-module 
 with a continuous semi-linear $G_K$-action such that $W^+_{\mathrm{dR}}\subseteq W_{\mathrm{dR}}:=B_{\mathrm{dR}}\otimes_{B_e}W_e$ 
 is a $G_K$-stable $B^+_{\mathrm{dR}}\otimes_{\mathbb{Q}_p}E$-lattice of $W_{\mathrm{dR}}$. The category of $E$-representations 
 of $G_K$ is embedded in the category of $E$-$B$-pairs by $V\mapsto W(V):=(B_e\otimes_{\mathbb{Q}_p}V, B^+_{\mathrm{dR}}\otimes_{\mathbb{Q}_p}V)$. 
 We say that an $E$-$B$-pair is split trianguline if $W$ is a successive extension of rank one $E$-$B$-pairs, i.e. $W$ has a filtration 
 $0\subseteq W_1\subseteq W_2\subseteq \cdots \subseteq W_{n-1}\subseteq W_n=W$ such that $W_i$ is a saturated sub $E$-$B$-pair 
 of $W$ and $W_i/W_{i-1}$ is a rank one $E$-$B$-pair for any $1\leqq i\leqq n$. We say that $W$ is trianguline 
 if $W\otimes_{E}E'$ is a split trianguline $E'$-$B$-pair for a finite extension $E'$ of $E$. We say that an $E$-representation $V$ is split trianguline 
 (resp. trianguline) if $W(V)$ is split trianguline (resp. trianguline). By these definitions, to study trianguline $E$-$B$-pairs, we first need to classify 
 rank one $E$-$B$-pairs and then we need to calculate extension class group of them, which were studied in [Co08] for $K=\mathbb{Q}_p$ and in 
 [Na09] for general $K$. In $\S$ 2.1, we recall these results which we need to study deformations of trianguline $E$-$B$-pairs.
In $\S$ 2.2, we define the Artin ring coefficient version of $B$-pairs. Let $\mathcal{C}_E$ be the category of Artin local $E$-algebras with residue field $E$. For any $A\in \mathcal{C}_E$, we say that $W:=(W_e,W^+_{\mathrm{dR}})$ is an $A$-$B$-pair if $W_e$ is a finite free $B_e\otimes_{\mathbb{Q}_p}A$-module with a continuous semi-linear $G_K$-action and 
 $W^+_{\mathrm{dR}}\subseteq W_{\mathrm{dR}}:=B_{\mathrm{dR}}\otimes_{B_e}W_e$ is a $G_K$-stable $B^+_{\mathrm{dR}}\otimes_{\mathbb{Q}_p}A$-lattice. We generalize some results of [Na09] to $A$-$B$-pairs.
 
 In $\S$ 3, we study two types of deformations for split trianguline $E$-$B$-pairs. In $\S$ 3.1, first we study usual 
 deformations for any $E$-$B$-pairs which are generalizations of homomorphismistic zero deformations
 of $p$-adic Galois representations. 
 Let $W$ be an $E$-$B$-pair and $A\in \mathcal{C}_E$. 
 Then we say that $(W_A,\iota)$ is a deformation of $W$ over $A$ if $W_A$ is an $A$-$B$-pair and $\iota:W_A\otimes_A E\isom W$ is an 
 isomorphism. Then we define the deformation functor of $W$, $D_W:\mathcal{C}_E\rightarrow (\mathrm{Sets})$ by 
 $D_W(A):=\{$equivalent classes of deformations $(W_A,\iota)$ of $W$ over $A$ $\}$. We prove pro-representability and formally smoothness and dimension formula 
 of $D_W$ (Corollary $\ref{13}$).  In $\S$ 3.2, we study the other type of deformations, i.e. trianguline deformations. Let $W$ be a split trianguline $E$-$B$-pair of rank $n$ and 
 $\mathcal{T}:0\subseteq W_1\subseteq W_2\subseteq \cdots \subseteq W_{n-1}\subseteq W_n=W$ be a fixed triangulation of $W$.  For any $A\in \mathcal{C}_E$, 
 we say that $(W_A, \iota, \mathcal{T}_A)$ is a trianguline deformation of $(W,\mathcal{T})$ over $A$ if $(W_A,\iota)$ is a deformation of $W$ 
 over $A$ and $\mathcal{T}_A:0\subseteq W_{1,A}\subseteq \cdots \subseteq W_{n-1,A}\subseteq W_{n,A}=W_A$ is a triangulation of $W_A$
( i.e. $W_{i,A}$ is a saturated sub $A$-$B$-pair of $W_A$ such that $W_{i,A}/W_{i-1,A}$ is a rank one $A$-$B$-pair for any $i$) such that 
$\iota(W_{i,A}\otimes_A E)=W_i$ for any $i$. Then we define the trianguline deformation functor of $(W,\mathcal{T})$, $D_{W,\mathcal{T}}:
\mathcal{C}_E\rightarrow (\mathrm{Sets})$ by $D_{W,\mathcal{T}}(A):=\{ $equivalent classes of trianguline deformations $(W_A,\iota,\mathcal{T}_A)$ 
of $(W,\mathcal{T})$ over $A$$\}$. We prove pro-representability and formally smoothness and dimension formula of this functor (Proposition $\ref{19}$).

In $\S$ 4, we define the notion of benign $E$-$B$-pairs, which forms a special good class of split trianguline $E$-$B$-pairs, and prove 
the main theorem of this article concerning to tangent spaces of the deformation rings of this class. 
In $\S$ 4.1, we define the notion of benign $E$-$B$-pairs. Let $W$ be a potentially crystalline $E$-$B$-pair of rank $n$ such that 
$W|_{G_L}$ is crystalline for a finite totally ramified abel extension $L$ of $K$ (we call such a representation a crystabelline representation).
We assume that $D_{\mathrm{cris}}^{L}(W):=(B_{\mathrm{cris}}\otimes_{B_e}W_e)^{G_L}=K_0\otimes_{\mathbb{Q_p}}Ee_1\oplus 
\cdots \oplus K_0\otimes_{\mathbb{Q}_p}Ee_n$ such that $K_0\otimes_{\mathbb{Q}_p}Ee_i$ are preserved by $(\varphi,\mathrm{Gal}(L/K))$-action 
and that $\varphi^f(e_i)=\alpha_ie_i$ for some $\alpha_i\in E^{\times}$, here $f:=[K_0:\mathbb{Q}_p]$ and $K_0$ is the maximal unramified extension of $\mathbb{Q}_p$ in $K$. 
We write $\{k_{1,\sigma}.k_{2,\sigma},\cdots,k_{n,\sigma}\}_{\sigma:K\hookrightarrow \bar{K}}$ the Hodge-Tate weight of $W$ such that 
$k_{1,\sigma}\geqq k_{2,\sigma}\geqq \cdots \geqq k_{n,\sigma}$ for any $\sigma:K\hookrightarrow \bar{K}$. 
Let $\mathfrak{S}_n$ be the $n$-th permutation group. Then, for any $\tau\in\mathfrak{S}_n$, we can define a filtration on $D^L_{\mathrm{cris}}(W)$ 
by $0\subseteq K_0\otimes_{\mathbb{Q}_p}Ee_{\tau(1)}\subseteq K_0\otimes_{\mathbb{Q}_p}Ee_{\tau(1)}\oplus K_0\otimes_{\mathbb{Q}_p}Ee_{\tau(2)}\subseteq 
\cdots \subseteq K_0\otimes_{\mathbb{Q}_p}Ee_{\tau(1)}\oplus \cdots \oplus K_0\otimes_{\mathbb{Q}_p}Ee_{\tau(n-1)}\subseteq D^L_{\mathrm{cris}}(W)$ by 
sub $E$-filtered $(\varphi, G_K)$-modules, where the filtration on $K_0\otimes_{\mathbb{Q}_p}Ee_{\tau(1)}\oplus \cdots \oplus K_0\otimes_{\mathbb{Q}_p}Ee_{\tau(i)}$ is 
the one induced from $L\otimes_{L_0}D^L_{\mathrm{cris}}(W)$. Then, by the $B$-pair version of ``weakly admissible = admissible" theorem, for any $\tau\in \mathfrak{S}_n$ 
we have a triangulation $\mathcal{T}_{\tau}:0\subseteq W_{\tau,1}\subseteq W_{\tau,2}\subseteq \cdots \subseteq W_{\tau,n}=W$ such that 
$W_{\tau,i}$ is potentially crystalline and $D^L_{\mathrm{cris}}(W_{\tau,i})\isom K_0\otimes_{\mathbb{Q}_p}Ee_{\tau(1)}\oplus \cdots \oplus K_0\otimes_{\mathbb{Q}_p}E
e_{\tau(i)}$

Under this condition, we define the notion of twisted benign $E$-$B$-pairs as follows.
\begin{defn}
Let $W$ be an $E$-$B$-pair of rank $n$ as above. Then we say that $W$ is twisted benign if $W$ satisfies the following;
\begin{itemize}
\item[(1)] For any $i\not= j$, we have $\alpha_i/\alpha_j\not= 1, p^f, p^{-f}$,
\item[(2)] For any $\sigma:K\hookrightarrow \bar{K}$, we have $k_{1,\sigma}>k_{2,\sigma}>\cdots >k_{n-1,\sigma}> k_{n,\sigma}$,
\item[(3)] For any $\tau\in \mathfrak{S}_n$ and $i$, Hodge-Tate weight of $W_{\tau,i}$ is $\{k_{1,\sigma},k_{2,\sigma},\cdots,k_{i,\sigma}\}_{\sigma :K\hookrightarrow \bar{K}}$.
\end{itemize}
\end{defn}

In $\S$ 4.2, we prove the main theorem of this article. Let $W$ be a twisted benign $E$-$B$-pair of rank $n$ as above. Then, for any $\tau\in \mathfrak{S}_n$, we can define 
the trianguline deformation functor $D_{W,\mathcal{T}_{\tau}}$. Let $R_W$ be the universal deformation ring of $D_W$ and 
$R_{W,\mathcal{T}_{\tau}}$ be the universal deformation ring of $D_{W,\mathcal{T}_{\tau}}$ for any $\tau\in \mathcal{S}_{\tau}$. 
Then $R_{W,\mathcal{T}_{\tau}}$ is a quotient ring of $R_W$. Let $t(R_{W})$ and $t(R_{W,\mathcal{T}_\tau})$  be the tangent spaces 
of $R_W$ and $R_{W,\mathcal{T}_{\tau}}$. Then, for any $\tau\in \mathfrak{S}_n$, $t(R_{W,\mathcal{T}_{\tau}})$ is a sub $E$-vector space of $t(R_W)$. 
The main theorem of this article is the following (Theorem $\ref{33}$), which is a generalization of Chenevier's theorem in [Ch09] in the 
case of $K=\mathbb{Q}_p$.
\begin{thm}
 Let $W$ be a twisted benign $E$-$B$-pair of rank $n$. Then we 
 have $$\sum_{\tau\in \mathfrak{S}_n}t(R_{W,\mathcal{T}_{\tau}})=t(R_W).$$
 \end{thm}
 
 This theorem is a crucial local results for applications to some Zariski density theorems of 
 local or global $p$-adic representations. In fact, by using 
 this theorem in the case of $K=\mathbb{Q}_p$, Chenevier, in  [Ch09], proved Zariski density of unitary automorphic Galois representations 
 in deformation spaces of three dimensional self dual $p$-adic representations of $G_{F}$ for any CM field $F$
 in which $p$ splits completely. In the next article, by using this main theorem, the author will prove Zariski density of crystalline 
 representations in deformation spaces of two dimensional $p$-adic representations for any $p$-adic fields, which is a generalization 
 of Colmez's and Kisin's theorem ([Co08], [Ki08b]) in the case of $K=\mathbb{Q}_p$.
 The author of this article hopes that,  by using this theorem, we can prove many kinds of Zariski density theorems for 
 general $p$-adic field cases or general global (CM or totally real) field cases.

\subsection*{Notation.}
Let $p$ be a prime number. $K$ be a finite extension of $\mathbb{Q}_p$, $\bar{K}$  a 
fixed algebraic closure of $K$, $K_0$ the maximal unramified extension of $\mathbb{Q}_p$ in $K$, 
$K^{\mathrm{nor}}$ the Galois closure of $K$ in $\bar{K}$. Let $G_K:=\mathrm{Gal}(\bar{K}/K)$ be the 
absolute Galois group of $K$ equipped with profinite topology. $\mathcal{O}_K$ is the integer ring of $K$, $\pi_K\in\mathcal{O}_K$ is 
a uniformizer of $K$, $k:=\mathcal{O}_K/\pi_K\mathcal{O}_K$ is the residue field 
of $K$, $q=p^f:=\sharp k$ is the order of $k$, $\chi:G_K\rightarrow \mathbb{Z}_p^{\times}$ is the p-adic cyclotomic homomorphism 
(i.e. $g(\zeta_{p^n})=\zeta_{p^n}^{\chi(g)}$ for any $p^n$-th roots of unity and for any $g\in G_K$). Let $\mathbb{C}_p:=\hat{\bar{K}}$ 
be the $p$-adic completion of $\bar{K}$, which is an algebraically closed $p$-adically complete field, and $\mathcal{O}_{\mathbb{C}_p}$ is its integer ring. We denote $v_p$ the normalized valuation on $\mathbb{C}_p^{\times}$ such that 
$v_p(p)=1$. Let $E$ be a finite extension of $\mathbb{Q}_p$ such that 
$K^{\mathrm{nor}}\subseteq E$. In this paper, we write $E$ as a coefficient of representations. 
We put $\mathcal{P}:=\{\sigma:K\hookrightarrow \bar{K}\}=\{\sigma:K\hookrightarrow E\}$ 
Let $\chi_{\mathrm{LT}}:G_K\rightarrow \mathcal{O}_K^{\times}\hookrightarrow \mathcal{O}_E^{\times}$ be 
the Lubin-Tate homomorphism associated with the fixed uniformizer $\pi_K$. Let $\mathrm{rec}_K:K^{\times}
\rightarrow G_K^{\mathrm{ab}}$ be the reciprocity map of local class field theory such that 
$\mathrm{rec}_K(\pi_K)$ is a lifting of the inverse of $q$-th power Frobenius on $k$, then 
$\chi_{\mathrm{LT}}\circ \mathrm{rec}_K:K^{\times}\rightarrow \mathcal{O}_K^{\times}$ satisfies 
$\chi_{\mathrm{LT}}\circ\mathrm{rec}_K(\pi_K)=1$ and $\chi_{\mathrm{LT}}\circ\mathrm{rec}_K|_{\mathcal{O}_K^{\times}}
=\mathrm{id}_{\mathcal{O}_K^{\times}}$.
\subsection*{Acknowledgement.}
The author would like to thank Kenichi Bannai for reading the first version of this article and for helpful comments. 
The author is supported by JSPS Core-to-Core program 18005 whose representative is Makoto Matsumoto.
 This work is also supported in part by KAKENHI (21674001).

\section{$B$-pairs.}

\subsection{$B$-pairs.}
In this subsection, first we recall the definition of $E$-$B$-pairs ([Be08],  [Na09]) and then recall fundamental 
properties of them established in [Na09]. 
First, we recall some $p$-adic period rings ([Fo94]) which we need for defining $B$-pairs.
Let $\widetilde{\mathbb{E}}^+:=\varprojlim_n\mathcal{O}_{\mathbb{C}_p}\isom
\varprojlim_n\mathcal{O}_{\mathbb{C}_p}/p\mathcal{O}_{\mathbb{C}_p}$, where the limits are taken with respect to $p$-th power maps. It is known that $\widetilde{\mathbb{E}}^+$ is a complete valuation ring of homomorphismistic $p$ whose valuation
is defined by $\mathrm{val}(x):=v_{p}(x^{(0)})$  (here $x=(x^{(n)})\in \varprojlim_n\mathcal{O}_{\mathbb{C}_p}$).
We fix a system of $p^n$-th roots of unity $\{\varepsilon^{(n)}\}_{n\ge 0}$ 
such that $\varepsilon^{(0)}=1$, $(\varepsilon^{(n+1) })^p=\varepsilon^{(n)}$, 
$\varepsilon^{(1)}\not= 1$. Then $\varepsilon:=(\varepsilon^{(n)})$ is an element of $\widetilde{\mathbb{E}}^+$ such that $\mathrm{val}(\varepsilon -1)=p/(p-1)$. $G_K$ acts on this ring in natural way. We put  $\widetilde{\mathbb{A}}^+:=W(\widetilde{\mathbb{E}}^+)$, where, for a perfect ring $R$,
$W(R)$ is the Witt ring of $R$. We put $\widetilde{\mathbb{B}}^+:=\widetilde{\mathbb{A}}^+[\frac{1}{p}]$. This ring also has natural $G_K$-action and Frobenius 
action $\varphi$. Then we have a $G_K$-equivariant surjection $\theta:\widetilde{\mathbb{A}}^+
\rightarrow \mathcal{O}_{\mathbb{C}_p}:\sum_{k=0}^{\infty}p^k[x_k]\mapsto \sum_{k=0}^{\infty} p^k x_k^{(0)}$, where $[\,\, ]:\widetilde{\mathbb{E}}^+\rightarrow \widetilde{\mathbb{A}}^+$ is the
Teichm$\ddot{\mathrm{u}}$ller homomorphism.
By inverting $p$, we get a surjection $\widetilde{\mathbb{B}}^+\rightarrow \mathbb{C}_p$.
We put $B_{\mathrm{dR}}^+:=\varprojlim_n\widetilde{\mathbb{B}}^+/(\mathrm{Ker}(\theta))^n$, which is
a complete discrete valuation ring with residue field $\mathbb{C}_p$ and 
is equipped with the projective limit topology of the $p$-adic topology 
on $\widetilde{\mathbb{B}}^+/(\mathrm{Ker}(\theta))^n$ whose lattice is the image of $\tilde{\mathbb{A}}^+\rightarrow 
\tilde{\mathbb{B}}^+/(\mathrm{Ker}(\theta))^n$. 
Let $A_{\mathrm{cris}}$ be the $p$-adic completion of $\widetilde{\mathbb{A}}^+[\frac{([\tilde{p}]-p)^n]}{n!}]_{n\geqq 1}$, where $\tilde{p}:=(p^{(n)})$ is an element in $\widetilde{\mathbb{E}}^+$ such that $p^{(0)}=p, (p^{(n+1) })^p=p^{(n)}$. We put $B_{\mathrm{cris}}^+:=A_{\mathrm{cris}}[\frac{1}{p}]$ equipped with $p$-adic topology whose lattice is $A_{\mathrm{cris}}$. $A_{\mathrm{cris}}$ and  
$B^+_{\mathrm{cris}}$ have $G_K$-actions and Frobenius actions $\varphi$. We have a natural 
$G_K$-equivariant embedding $K\otimes_{K_0}B^+_{\mathrm{cris}}\hookrightarrow B^+_{\mathrm{dR}}$.
If we put $t:=\mathrm{log}([\varepsilon])=\sum_{n=1}^{\infty}(-1)^{n-1}\frac{([\varepsilon]-1)^n}{n}$, then we can see that $t\in A_{\mathrm{cris}}$, 
$\varphi(t)=pt, g(t)=\chi(g)t$ for any $g\in G_K$ and $\mathrm{Ker}(\theta)=tB_{\mathrm{dR}}^+
\subset B_{\mathrm{dR}}^+$ is the maximal ideal of $B^+_{\mathrm{dR}}$. If we put $B_{\mathrm{cris}}:=B_{\mathrm{cris}}^+
[\frac{1}{t}], B_{\mathrm{dR}}:=B_{\mathrm{dR}}^+[\frac{1}{t}]$, we have a natural embedding 
$K\otimes_{K_0}B_{\mathrm{cris}}\hookrightarrow B_{\mathrm{dR}}$. 
We put $B_e:=B_{\mathrm{cris}}^{\varphi =1}$ on which we equipped the inductive limit topology of 
$B_e=\cup_{n}(\frac{1}{t^n}B^{+}_{\mathrm{cris}})^{\varphi=1}$, where each $(\frac{1}{t^n}B^{+}_{\mathrm{cris}})^{\varphi=1}
=\frac{1}{t^n}B^{+ \varphi=p^n}_{\mathrm{cris}}$ is 
equipped with $p$-adic topology. We put $\mathrm{Fil}^i B_{\mathrm{dR}}:=t^i B_{\mathrm{dR}}^+$ for any $i\in\mathbb{Z}$.

In this paper, we fix a coefficient field of $p$-adic representations or $B$-pairs of $G_K$.
So we start in this subsection by recalling an $E$-coefficient version of $p$-adic representations and $B$-pairs.
\begin{defn}\label{a}
An $E$-representation of $G_K$ is a finite dimensional $E$-vector space $V$ with a continuous $E$-linear action of $G_K$. We call $E$-representation for simplicity when there is no risk 
of confusion about $K$.
\end{defn}
\begin{defn} \label{b}
A pair $W:=(W_e,W^+_{\mathrm{dR}})$ is an $E$-$B$-pair if 
\begin{itemize}
\item[(1)] $W_e$ is a finite $B_e\otimes_{\mathbb{Q}_p}E$-module which is free over $B_e$ with 
a continuous semi-linear $G_K$-action.
\item[(2)] $W^+_{\mathrm{dR}}$ is a $G_K$-stable finite sub $B^+_{\mathrm{dR}}\otimes_{\mathbb{Q}_p}E$-module 
of $B_{\mathrm{dR}}\otimes_{B_e}W_e$ which generates $B_{\mathrm{dR}}\otimes_{B_e}W_e$ as a $B_{\mathrm{dR}}$-module.
\end{itemize}
\end{defn}
Then, by Lemma 1.7, 1.8 of [Na09], $W_e$ is a free $B_e\otimes_{\mathbb{Q}_p}E$-module and 
$W^+_{\mathrm{dR}}$ is a free $B^+_{\mathrm{dR}}\otimes_{\mathbb{Q}_p}E$-module. We define the rank of $W$ by $\mathrm{rank}(W):=\mathrm{rank}_{B_e\otimes_{\mathbb{Q}_p}E}(W_e)$.
For $E$-$B$-pairs $W_1:=(W_{1,e},W^+_{1,\mathrm{dR}})$ and $W_2:=(W_{2,e}, W^+_{2,\mathrm{dR}})$, 
we define the tensor product of $W_1$ and $W_2$ by  $W_1\otimes W_2:=(W_{1,e}\otimes_{B_e\otimes_{\mathbb{Q}_p}E}W_{2,e}, W^+_{1,\mathrm{dR}}\otimes_{B^+_{\mathrm{dR}}\otimes_{\mathbb{Q}_p}E}W^+_{2,\mathrm{dR}})$ and define the dual of $W_1$ 
by $W_1^{\vee}:=(\mathrm{Hom}_{B_e\otimes_{\mathbb{Q}_p}E}(W_{1,e}, B_{e}\otimes_{\mathbb{Q}_p}E), \{f\in \mathrm{Hom}_{B_{\mathrm{dR}}\otimes_{\mathbb{Q}_p}E}(B_{\mathrm{dR}}
\otimes_{B_e}W_{1,e}, B_{\mathrm{dR}}\otimes_{\mathbb{Q}_p}E)| f(W^+_{1,\mathrm{dR}})\subseteq B^+_{\mathrm{dR}}\otimes_{\mathbb{Q}_p}E\})$  (remark: there is a mistake in the 
Definition 1.9 of [Na09]).
The category of $E$-$B$-pairs is not an abelian category, i.e. an inclusion $W_1\hookrightarrow W_2$ in general 
does not have a quotient in the category of $E$-$B$-pairs. But we can always take the saturation $W^{\mathrm{sat}}_{1}:=(W^{\mathrm{sat}}_{1.e}, W^{+,\mathrm{sat}}_{1,\mathrm{dR}})$ such that $W^{\mathrm{sat}}_1$ sits in 
$W_1\hookrightarrow W^{\mathrm{sat}}_1\hookrightarrow W_2$ and $W_{1,e}=W^{\mathrm{sat}}_{1,e}$ and $W_2/W^{\mathrm{sat}}_{1}$ is an $E$-$B$-pair (Lemma 1.14 of [Na09]). We say that an inclusion $W_1\hookrightarrow W_2$ is saturated if $W_2/W_1$ is an $E$-$B$-pair.

Next, we recall $p$-adic Hodge theory for $B$-pairs. Let $W=(W_e,W^+_{\mathrm{dR}})$ be an $E$-$B$-pair. 
We define $D_{\mathrm{cris}}(W):=(B_{\mathrm{cris}}\otimes_{B_e}W_e)^{G_K}$, $D^L_{\mathrm{cris}}(W):=
(B_{\mathrm{cris}}\otimes_{B_e}W_e)^{G_L}$ for any finite extension $L$ of $K$, $D_{\mathrm{dR}}(W):=(B_{\mathrm{dR}}\otimes_{B_e}
W_e)^{G_K}$, $D_{\mathrm{HT}}(W):=(B_{\mathrm{HT}}\otimes_{\mathbb{C}_p}(W^+_{\mathrm{dR}}/tW^+_{\mathrm{dR}}))^{G_K}$, here 
$B_{\mathrm{HT}}:=\mathbb{C}_p[T, T^{-1}]$ where $G_K$ acts by $g(aT^i):=\chi(g)^ig(a)T^i$ for any $g\in G_K, a\in \mathbb{C}_p, i\in \mathbb{Z}$.
\begin{defn}\label{c}
We say that $W$ is crystalline (resp. de Rham, resp. Hodge-Tate) if 
$\mathrm{dim}_{K_{\ast}}D_{\mathrm{\ast}}(W)=[E:\mathbb{Q}_p]\mathrm{rank}(W)$ for $\ast= \mathrm{cris}$ (resp. $\ast=\mathrm{dR}$, resp. 
$\ast=\mathrm{HT}$), where $K_{\ast}=K_0$ when $\ast=\mathrm{cris}$ and $K_{\ast}=K$ when $\ast=\mathrm{dR},\mathrm{HT}$.
We say that $W$ is potentially crystalline if $\mathrm{dim}_{L_0}(D^L_{\mathrm{cris}}(W))=[E:\mathbb{Q}_p]\mathrm{rank}(W)$ for 
a finite extension $L$ of $K$, where $L_0$ is the maximal unramified extension of $\mathbb{Q}_p$ in $L$.
\end{defn}

\begin{defn}\label{d}
Let $L$ be a finite Galois extension of $K$ and $G_{L/K}:=\mathrm{Gal}(L/K)$.
We say that $D$ is an $E$-filtered ($\varphi,G_{L/K}$)-module over $K$ if 
\begin{itemize}
\item[(1)] $D$ is a finite free $L_0\otimes_{\mathbb{Q}_p} E$-module with a $\varphi$-semi-linear 
action $\varphi_D$ and a semi-linear action of $G_{L/K}$ such that 
$\varphi_D:D\isom D$ is an isomorphism and that $\varphi_D$ and  $G_{L/K}$-action commute, where ($\varphi$-)semi-linear means that 
$\varphi_D(a\otimes b\cdot x)=\varphi(a)\otimes b\cdot \varphi_D(x)$, $g(a\otimes b\cdot x)=g(a)\otimes b\cdot g(x)$ for 
any $a\in L_0, b\in E, x\in D,  g\in G_{L/K}$.
\item[(2)] $D_L:=L\otimes_{L_0}D$ has a separated and exhausted decreasing filtration $\mathrm{Fil}^iD_L$ by 
sub $L\otimes_{\mathbb{Q}_p}E$-modules such that $G_{L/K}$-action on $D_L$ preserves this filtration.
\end{itemize}
\end{defn}

Let $W$ be a potentially crystalline $E$-$B$-pair such that $W|_{G_L}$ is crystalline 
for a finite Galois extension $L$ of $K$. Then we define an $E$-filtered ($\varphi, G_{L/K}$)-module 
structure on $D^L_{\mathrm{cris}}(W)$ as follows. First, $D^L_{\mathrm{cris}}(W)$ has a Frobenius action 
induced from that on $B_{\mathrm{cris}}$ and has a $G_{L/K}$-action induced from those on $B_{\mathrm{cris}}$ and 
$W_e$. 
We define a filtration on $L\otimes_{L_0}D^L_{\mathrm{cris}}(W)=L\otimes_{K}D_{\mathrm{dR}}(W)$ by 
$\mathrm{Fil}^i(L\otimes_{L_0}D^{L}_{\mathrm{cris}}(W)):=(L\otimes_{K}D_{\mathrm{dR}}(W))\cap t^iW^+_{\mathrm{dR}}$.

Let $D:=L_0e$ be a rank one $\mathbb{Q}_p$-filtered ($\varphi, G_{L/K}$)-module with a base $e$. then 
we define $t_N(D):=v_p(\alpha)$ where $\varphi_D(e)=\alpha\cdot e$ and define $t_H(D):=i$ such that 
$\mathrm{Fil}^iD_L/\mathrm{Fil}^{i+1}D_L\not= 0$. For general $D$ of rank $d$, we define $t_N(D):=t_N(\wedge^dD)$, 
$t_H(D):=t_H(\wedge^dD)$. Then we say that $D$ is weakly admissible if $t_N(D)=t_H(D)$ and $t_N(D')\geqq t_H(D')$ for any 
sub $\mathbb{Q}_p$-filtered ($\varphi,G_{L/K}$)-module $D'$ of $D$.
\begin{thm}\label{e}
Let $L$ be a finite Galois extension of $K$. Then we have the following results.
\begin{itemize}
\item[(1)] The functor $W\mapsto D^L_{\mathrm{cris}}(W)$ gives an equivalence of categories between 
the category of potentially crystalline $E$-$B$-pairs which are crystalline if restricted to $G_L$ and the 
category of $E$-filtered $(\varphi,G_{L/K})$-modules over $K$.
\item[(2)] By restricting the above functor to $E$-representations, the functor $V\mapsto D^{L}_{\mathrm{cris}}(V)$ 
gives an equivalence of categories between the category of potentially crystalline $E$-representations which are 
crystalline if restricted to $G_L$  and the category of weakly admissible $E$-filtered $(\varphi,G_{L/K})$-modules over $K$.
\end{itemize}
\end{thm}
\begin{proof}
See [Be08, Proposition 2.3.4 and Theorem 2.3.5] or [Na09 Theorem 1.18].
\end{proof}

Next, we define trianguline $E$-$B$-pairs, whose deformation theory we study in detail in this article. 
\begin{defn}\label{f}
Let $W$ be an $E$-$B$-pair of rank $n$. Then we say that $W$ is split trianguline if 
there exists a filtration $\mathcal{T}:0\subseteq W_1\subseteq W_2\subseteq \cdots \subseteq W_n=W$ by 
sub $E$-$B$-pairs such that $W_{i}$ is saturated in $W_{i+1}$ and $W_{i+1}/W_{i}$ is a rank one 
$E$-$B$-pair for any $i$. We say that $W$ is trianguline if $W\otimes_{E}E'$, base change of $W$  to $E'$, 
is a split trianguline $E'$-$B$-pair for a finite extension $E'$ of $E$. 
\end{defn} 

By this definition, for studying split trianguline $E$-$B$-pairs, it is important to classify rank one $E$-$B$-pairs 
and calculate extension classes by rank one objects, which were studied in the previous article by the author ([Na09]). 
Now, we recall some  results concerning these.

\begin{thm}\label{g}
There exists canonical one to one correspondence $\delta\mapsto W(\delta)$ between the set of 
continuous homomorphisms $\delta:K^{\times}\rightarrow E^{\times}$ and  the set of isomorphism classes
of rank one $E$-$B$-pairs. 
\end{thm}
\begin{proof}
See  Proposition 3.1 of [Co08] for $K=\mathbb{Q}_p$ and Theorem 1.45 of [Na09] for general $K$.
For construction of $W(\delta)$, see $\S$1.4 of [Na09].
\end{proof}

This correspondence is compatible with local class field theory, i.e. for a unitary homomorphism 
$\delta:K^{\times}\rightarrow \mathcal{O}_E^{\times}$ if we take  the $\tilde{\delta}:G^{\mathrm{ab}}_K\rightarrow \mathcal{O}_E^{\times}$ 
satisfying $\tilde{\delta}\circ \mathrm{rec}_K=\delta$, then we have $W(\delta)\isom W(E(\tilde{\delta}))$. 
This correspondence is also compatible with tensor products and with duals, i.e for homomorphisms $\delta_1, \delta_2:K^{\times}
\rightarrow E^{\times}$, we have  $W(\delta_1)\otimes W(\delta_2)\isom W(\delta_1\delta_2)$ and $W(\delta_1)^{\vee}
\isom W(\delta_1^{-1})$. There are important examples of rank one $E$-$B$-pairs which we recall now.
For any $\{k_{\sigma}\}_{\sigma\in\mathcal{P}}$ such that $k_{\sigma}\in\mathbb{Z}$, we define a homomorphism 
$\prod_{\sigma\in\mathcal{P}}\sigma(x)^{k_{\sigma}}:K^{\times}\rightarrow E^{\times}:y\mapsto \prod_{\sigma\in \mathcal{P}}
\sigma(y)^{k_{\sigma}}$.  Then we have $W(\prod_{\sigma\in\mathcal{P}}\sigma(x)^{k_{\sigma}})=(B_e\otimes_{\mathbb{Q}_p}E, 
\oplus_{\sigma\in\mathcal{P}}t^{k_{\sigma}}B^+_{\mathrm{dR}}\otimes_{K,\sigma}E)$ ( Lemma 2.12 of [Na09]).
 Let  $N_{K/\mathbb{Q}_p}:K^{\times}\rightarrow \mathbb{Q}_p^{\times}$ be the norm map and 
 $|-|:\mathbb{Q}_p^{\times}\rightarrow \mathbb{Q}^{\times}\hookrightarrow E^{\times}$ be the $p$-adic absolute value 
 such that $|p|=\frac{1}{p}$, and we define $|N_{K/\mathbb{Q}_p}(x)|:K^{\times}\rightarrow E^{\times}$  the composite
 of $N_{K/\mathbb{Q}_p}$ and $|-|$. Then we have $W(|N_{K/\mathbb{Q}_p}(x)|\prod_{\sigma\in \mathcal{P}}\sigma(x))=W(E(\chi))$, which is the $E$-$B$-pair
 associated to $p$-adic cyclotomic homomorphism.  These homomorphisms play special roles in the calculation of Galois cohomology of 
 $E$-$B$-pairs. 
 Next, we recall the definition and some properties of Galois cohomology of $E$-$B$-pairs. 
 For an $E$-$B$-pair $W:=(W_e, W^+_{\mathrm{dR}})$, we put $W_{\mathrm{dR}}:=B_{\mathrm{dR}}\otimes_{B_e}W_e$. 
 We have natural inclusions, $W_e\hookrightarrow W_{\mathrm{dR}}$, $W^+_{\mathrm{dR}}
 \hookrightarrow W_{\mathrm{dR}}$. Then we define the Galois cohomology $\mathrm{H}^i(G_K, W)$ of $W$ to be  
 the Galois cohomology of the complex $W_e\oplus W^+_{\mathrm{dR}}\rightarrow W_{\mathrm{dR}}:(x,y)\mapsto 
 x+y$ (see $\S$2.1 of [Na09] for the precise definition). Then as in the usual $p$-adic representation case, we have $\mathrm{H}^0(G_K, W)
 =\mathrm{Hom}_{G_K}(B_E, W)$ and $\mathrm{H}^1(G_K, W)=\mathrm{Ext}^1(B_E, W)$, where $B_E:=(B_e\otimes_{\mathbb{Q}_p}E, 
 B^+_{\mathrm{dR}}\otimes_{\mathbb{Q}_p}E)$ is the trivial $E$-$B$-pair and $\mathrm{Hom}_{G_K}(-,-)$ is the group of homomorphisms
 of $E$-$B$-pairs and $\mathrm{Ext}^1(-,-)$ is the extension class group in the category of $E$-$B$-pairs. If 
 $V$ is an $E$-representation of $G_K$, then we have $\mathrm{H}^i(G_K, V)\isom \mathrm{H}^i(G_K, W(V))$, which follows from 
 Bloch-Kato's fundamental short exact sequence $0\rightarrow \mathbb{Q}_p\rightarrow B_e\oplus B^+_{\mathrm{dR}}
 \rightarrow B_{\mathrm{dR}}\rightarrow 0$.
 Moreover, we have the following theorem, Euler-Poincar$\mathrm{\acute{e}}$ homomorphismistic formula and Tate duality theorem for 
 $B$-pairs.
 \begin{thm}\label{h}
 Let $W$ be an $E$-$B$-pair. Then$:$
 \begin{itemize}
 \item[(1)] for $i=0,1,2$, $\mathrm{H}^i(G_K, W)$ is finite dimensional over $E$ and $\mathrm{H}^i(G_K, W)=0$ for 
 $i\not= 0,1,2$.
 \item[(2)] $\sum_{i=0}^2(-1)^i \mathrm{dim}_E\mathrm{H}^i(G_K, W)=[K:\mathbb{Q}_p] \mathrm{rank}(W)$.
 \item[(3)] let $W$ be a $\mathbb{Q}_p$-$B$-pair, Then, for any $i=0,1,2$, there is a natural perfect pairing defined by cup 
 product,
 \item[] \[
              \begin{array}{ll}
              \mathrm{H}^i(G_K, W)\times\mathrm{H}^{2-i}(G_K, W^{\vee}(\chi))&\rightarrow \mathrm{H}^2(G_K, W\otimes W^{\vee}(
 \chi)) \\
  &\rightarrow \mathrm{H}^2(G_K, W(\mathbb{Q}_p(\chi))\allowbreak \isom \mathbb{Q}_p,
  \end{array}
  \]
 \item[] where the last isomorphism is the trace map.
 \end{itemize}
 \end{thm}
 \begin{proof}
 See Theorem 4.3, 4.7 of [Liu08] and Theorem 2.8 of [Na09].
 \end{proof}
 By using these formulae, we have the following dimension formulae for rank one $E$-$B$-pairs.
 \begin{prop}\label{i}
 Let $\delta:K^{\times}\rightarrow E^{\times}$ be a continuous homomorphism. Then we have$:$
 \begin{itemize}
 \item[(1)] $\mathrm{H}^0(G_K, W(\delta))\isom E $ if $\delta=\prod_{\sigma\in \mathcal{P}}\sigma(x)^{k_{\sigma}}$ 
 such that $k_{\sigma}\in \mathbb{Z}_{\leqq 0}$ for any $\sigma\in \mathcal{P}$, and $\mathrm{H}^0(G_K, W(\delta))=0$ otherwise.
 \item[(2)] $\mathrm{H}^2(G_K, W(\delta))\isom E$ if $\delta=|N_{K/\mathbb{Q}_p}|\prod_{\sigma\in\mathcal{P}}\sigma(x)^{k_{\sigma}}$ 
 such that $k_{\sigma}\in \mathbb{Z}_{\geqq 1}$ for any $\sigma\in \mathcal{P}$, and $\mathrm{H}^2(G_K, W(\delta))=0$ otherwise.
 \item[(3)] $\mathrm{dim}_E\mathrm{H}^1(G_K, W(\delta))=[K:\mathbb{Q}_p] +1$ if $\delta=\prod_{\sigma\in \mathcal{P}}\sigma(x)^{k_{\sigma}}$ such that $k_{\sigma}\in \mathbb{Z}_{\leqq 0}$ for any $\sigma\in \mathcal{P}$ or $\delta=
 |N_{K/\mathbb{Q}_p}|\prod_{\sigma\in \mathcal{P}}\sigma(x)^{k_{\sigma}}$ such that $k_{\sigma}\in \mathbb{Z}_{\geqq 1}$ for any $\sigma\in \mathcal{P}$, and $\mathrm{dim}_E\mathrm{H}^1(G_K, W(\delta))=[K:\mathbb{Q}_p]$ otherwise.
 \end{itemize}
 \end{prop}
 \begin{proof}
 See Theorem 2.9 and Theorem 2.22 of [Co08] for $K=\mathbb{Q}_p$. For general $K$, the results may be proved by using 
 Proposition 2.14 and Proposition 2.15 of [Na09] and Tate duality for $B$-pairs.
 \end{proof}

\subsection{$B$-pairs over Artin local rings}
Let $\mathcal{C}_E$ 
be the category of Artin local $E$-algebra $A$ with residue field $E$. The morphisms of $\mathcal{C}_E$ are given by 
local $E$-algebra homomorphisms. For any $A\in \mathcal{C}_E$, let $m_A$ be the maximal ideal of $A$. We define the $A$-coefficient versions of $B$-pairs as follows.
\begin{defn}\label{1}
We call a pair $W:=(W_e, W^+_{\mathrm{dR}})$ an $A$-$B$-pair of $G_K$ if 
\begin{itemize}
\item[(1)] $W_e$ is a finite $B_e\otimes_{\mathbb{Q}_p}A$-module which is flat over $A$ 
and is free over $B_e$, with a continuous semi-linear $G_K$-action.
\item[(2)] $W^+_{\mathrm{dR}}$ is a finite generated sub $B^+_{\mathrm{dR}}\otimes_{\mathbb{Q}_p}A$-module of $B_{\mathrm{dR}}\otimes_{B_e}W_e$ which is preserved by $G_K$-action and which generates 
$B_{\mathrm{dR}}\otimes_{B_e}W_e$ as a $B_{\mathrm{dR}}\otimes_{\mathbb{Q}_p}A$-module such that 
$W^+_{\mathrm{dR}}/tW^+_{\mathrm{dR}}$ is flat over $A$.
\end{itemize}
For an $A$-$B$-pair $W:=(W_e,W^+_{\mathrm{dR}})$, we put $W_{\mathrm{dR}}:=B_{\mathrm{dR}}
\otimes_{B_e}W_e$.
\end{defn}
We simply call an $A$-$B$-pair if there is no risk of confusing about $K$.

\begin{lemma}\label{2}
Let $W:=(W_e,W^+_{\mathrm{dR}})$ be an $A$-$B$-pair. Then 
$W_e$ is a finite free $B_e\otimes_{\mathbb{Q}_p}A$-module and 
$W^+_{\mathrm{dR}}$ is a finite free $B^+_{\mathrm{dR}}\otimes_{\mathbb{Q}_p}A$-module 
and $W^+_{\mathrm{dR}}/tW^+_{\mathrm{dR}}$ is a finite free $\mathbb{C}_p\otimes_{\mathbb{Q}_p}A$-module.
\end{lemma}
\begin{proof}
First, we prove for $W_e$. Because the sub module $m_A W_e\subseteq W_e$ is a $G_K$-stable finite generated 
torsion free $B_e$-module, $m_AW_e$ is a finite free $B_e$-module by Lemma 2.4 of [Ke04]. 
Then, by Lemma 2.1.4 of [Be08], the cokernel $W_e\otimes_A E$ is also a finite free $B_e$-module 
(with an $E$-action). Then, by Lemma 1.7 of [Na09], $W_e\otimes_AE$ is a finite free $B_e\otimes_{\mathbb{Q}_p}E$-module of 
some rank $n$.
We take a $B_e\otimes_{\mathbb{Q}_p}A$-linear morphism $f:(B_e\otimes_{\mathbb{Q}_p}A)^n\rightarrow W_e$ which 
is a lift of a $B_e\otimes_{\mathbb{Q}_p}E$-linear isomorphism $(B_e\otimes_{\mathbb{Q}_p}E)^n\isom W_e\otimes_A E$.
Then, by Nakayama's lemma, $f$ is surjective. Because $W_e$ is $A$-flat, we have $\mathrm{Ker}(f)\otimes_AE=0$, 
so $\mathrm{Ker}(f)=0$. So $W_e$ is a free $B_e\otimes_{\mathbb{Q}_p}A$-module.

Next, we prove that $W^+_{\mathrm{dR}}$ is a free $B^+_{\mathrm{dR}}\otimes_{\mathbb{Q}_p}A$-module.
Because $W_e$ is a free $B_e\otimes_{\mathbb{Q}_p}A$-module, $W_{\mathrm{dR}}$ is a free $B_{\mathrm{dR}}\otimes_{\mathbb{Q}_p}A$-module, in particular this is flat over $A$. Because $W^+_{\mathrm{dR}}/tW^+_{\mathrm{dR}}$ is a flat $A$-module, $W_{\mathrm{dR}}/W^+_{\mathrm{dR}}$ is also  a flat $A$-module. So $W^+_{\mathrm{dR}}$ is also flat over $A$. 
By the $A$-flatness of $W_{\mathrm{dR}}/W^+_{\mathrm{dR}}$, we have inclusion 
$W^+_{\mathrm{dR}}\otimes_A E\hookrightarrow W_{\mathrm{dR}}\otimes_{A} E$, so 
$W^{+}_{\mathrm{dR}}\otimes_A E$ is a finite generated torsion free $B^+_{\mathrm{dR}}$-module, so 
free $B^+_{\mathrm{dR}}$-module. By using these and Lemma 1.8 of [Na09], in the same way as in the case of $W_e$, we can show that 
$W^+_{\mathrm{dR}}$ is a free $B^+_{\mathrm{dR}}\otimes_{\mathbb{Q}_p}A$-module.
$\mathbb{C}_p\otimes_{\mathbb{Q}_p}A$-freeness of $W^+_{\mathrm{dR}}/tW^+_{\mathrm{dR}}$ follows from 
$B^+_{\mathrm{dR}}\otimes_{\mathbb{Q}_p}A$-freeness of $W^+_{\mathrm{dR}}$.

\end{proof}
By using this lemma, we can define the rank of $A$-$B$-pair.
\begin{defn}
Let $W=(W_e, W^+_{\mathrm{dR}})$ be an $A$-$B$-pair. Then we define the rank of $W$ as 
$\mathrm{rank}(W):=\mathrm{rank}_{B_e\otimes_{\mathbb{Q}_p}A}(W_e)$.
\end{defn}

\begin{defn}\label{3}
Let $f:A\rightarrow A'$ be a morphism in $\mathcal{C}_E$ and $W=(W_e,W_{\mathrm{dR}})$ be an $A$-$B$-pair.
Then we define the base change of $W$ to $A'$ by $W\otimes_A A':=(W_e\otimes_A A',W^+_{\mathrm{dR}}\otimes_A A')$.
By using lemma $\ref{2}$, we can easily see that this is an $A'$-$B$-pair.
\end{defn}

\begin{defn}\label{4}
Let $W_1=(W_{e,1},W^+_{\mathrm{dR},1})$, $W_2=(W_{e,2}, W^+_{\mathrm{dR},2})$ be $A$-$B$-pairs.
Then we define the tensor product of $W_1$ and $W_2$ by $W_1\otimes W_2:=(W_{e,1}\otimes_{B_e\otimes_{\mathbb{Q}_p}A}W_{e,2}, 
W^+_{\mathrm{dR},1}\otimes_{B^+_{\mathrm{dR}}\otimes_{\mathbb{Q}_p}A}W^+_{\mathrm{dR},2})$ , 
and define the dual of $W_1$ by $W_1^{\vee}:=(\mathrm{Hom}_{B_e\otimes_{\mathbb{Q}_p}A}(W_{e,1},\allowbreak B_e\otimes_{\mathbb{Q}_p}A), 
W^{+,\vee}_{\mathrm{dR},1})$. Here, $W^{+,\vee}_{\mathrm{dR},1}:=\{f\in\mathrm{Hom}_{B_{\mathrm{dR}}\otimes_{\mathbb{Q}_p}A}(
W_{\mathrm{dR},1}, B_{\mathrm{dR}}\otimes_{\mathbb{Q}_p}A)|f(W^{+}_{\mathrm{dR},1})\subseteq B^+_{\mathrm{dR}}\otimes_{\mathbb{Q}_p}A\}$. By using lemma $\ref{2}$, we can easily see that these are $A$-$B$-pairs.
\end{defn}

Next, we classify rank one $A$-$B$-pairs.
Let $\delta:K^{\times}\rightarrow A^{\times}$ be a continuous homomorphism.
Then we define the rank one $A$-$B$-pair $W(\delta)$ as follows.
Let $\bar{u}\in E^{\times}$ be the reduction of $u:=\delta(\pi_K)$.
We define a homomorphism $\delta_0:K^{\times}\rightarrow A^{\times}$ such that $\delta_0|_{\mathcal{O}_K^{\times}}
=\delta|_{\mathcal{O}_K^{\times}}$, $\delta_0(\pi_K)=u/\bar{u}$. Then, because $u/\bar{u}\in 1+m_A$ ,  $(u/\bar{u})^{p^n}$ 
converges to $1\in A^{\times}$. So if we fix an isomorphism $K^{\times}\isom \mathcal{O}_K^{\times}\times \mathbb{Z}
:v\pi_K^n\mapsto (v,n)$ (here $v\in\mathcal{O}_K^{\times}$), then $\delta_0$ extends to $\delta_0':\mathcal{O}_K^{\times}\times 
\hat{\mathbb{Z}}\rightarrow \mathcal{O}_K^{\times}\times \mathbb{Z}_p\rightarrow A^{\times}$. Hence, by local class field theory, there exists unique homomorphism $\tilde{\delta}_0:G_K^{\mathrm{ab}}\rightarrow A^{\times}$ such 
that $\delta_0=\tilde{\delta}_0\circ\mathrm{rec}_K$, here $\mathrm{rec}_K:K^{\times}\rightarrow G_K^{\mathrm{ab}}$ is the reciprocity map as in Notation. We define the 
$\mathrm{\acute{e}}$tale rank one $A$-$B$-pair $W(A(\tilde{\delta}_0))$, which is the $A$-$B$-pair associated to the rank one $A$-representation $A(\tilde{\delta}_0)$.
Next, we define a non-$\mathrm{\acute{e}}$tale rank one $A$-$B$-pair by using $\bar{u}\in E^{\times}$. 
For this, first we define a rank one $E$-filtered $\varphi$-module $D_{\bar{u}}:=K_0\otimes_{\mathbb{Q}_p}Ee_{\bar{u}}$ such that 
$\varphi^f(e_{\bar{u}}):=\bar{u}e_{\bar{u}}$ and $\mathrm{Fil}^0(K\otimes_{K_0}D_{\bar{u}}):=K\otimes_{K_0}D_{\bar{u}}$, $\mathrm{Fil}^1(K\otimes_{K_0}D_{\bar{u}}):=0$.
From this, we get the rank one crystalline $E$-$B$-pair $W(D_{\bar{u}})$ such that $D_{\mathrm{cris}}(W(D_{\bar{u}}))\isom D_{\bar{u}}$ which is 
pure of slope $\frac{v_p(\bar{u})}{f}$.
By tensoring these, we define the rank one $A$-$B$-pair $W(\delta)$ by $W(\delta):=(W(D_{\bar{u}})\otimes_EA)\otimes W(A(\tilde{\delta}_0))$,
 which is pure of slope $\frac{v_p(\bar{u})}{f}$.
 
 The following proposition is the $A$-coefficient version of Theorem 1.45 of [Na09].
 
 \begin{prop}\label{5}
 This construction $\delta\mapsto W(\delta)$ does not depend on the 
 choice of uniformizer $\pi_K$ and gives a bijection between 
 the set of continuous homomorphisms $\delta:K^{\times}\rightarrow A^{\times}$ and 
 the set of isomorphism classes of rank one $A$-$B$-pairs.
 \end{prop}
 \begin{proof}
 The independence of the choice of uniformizer and the injection is proved in the same 
 way as in the proof of Theorem 1.45 of [Na09]. We prove the surjection. Let $W$ be a rank one $A$-$B$-pair. Then 
 as an $E$-$B$-pair, $W$ is a successive extension of rank one $E$-$B$-pair $W\otimes_A E$.
 $W\otimes_AE$ is pure of slope $\frac{n}{fe_E}$ for some $n\in \mathbb{Z}$ by Lemma 1.42 of [Na09]. 
 So, by Theorem 1.6.6 of [Ke08], $W$ is also pure of slope $\frac{n}{fe_E}$.  
 We define the rank one $E$-filtered $\varphi$-module $D_{\pi_K^n}:=K_0\otimes_{\mathbb{Q}_p}Ee_{\pi_K^n}$ 
 in the same way as in $D_{\bar{u}}$. Then $W\otimes (W(D_{\pi_K^n}) \otimes_E A)^{\vee}$ is pure of slope zero by Lemma 
 1.34 of [Na09]. So there exists $\tilde{\delta}':G_K^{\mathrm{ab}}\rightarrow A^{\times}$ such that 
 $W\otimes (W(D_{\pi^n})\otimes_E A)^{\vee}\isom W(A(\tilde{\delta}'))$. If we put $\delta':=\tilde{\delta}'\circ\mathrm{rec}_K:
 K^{\times}\rightarrow A^{\times}$ and put $\delta:K^{\times}\rightarrow A^{\times}$ such that $\delta|_{\mathcal{O}_K^{\times}}
 :=\delta'|_{\mathcal{O}_K^{\times}}$ and $\delta(\pi_K):=\delta'(\pi_K)\pi_K^n$, then we have an isomorphism 
 $W\isom W(\delta)$ which can be easily seen from the construction of $W(\delta)$.

 \end{proof}
 By local class field theory, we have a canonical bijection $\delta\mapsto A(\tilde{\delta})$ from  the set of unitary continuous homomorphisms 
 from $K^{\times}$ to $A^{\times}$ (where unitary means that the reduction of the image of $\delta$ is contained in $\mathcal{O}_E^{\times}$) 
  to the set of isomorphism class of rank one $A$-representations of $G_K$, where $\tilde{\delta}:G_K^{ab}\rightarrow 
  A^{\times}$ is the continuous homomorphism such that $\delta=\tilde{\delta}\circ\mathrm{rec}_K$.
  By the definition of $W(\delta)$  and by the above proof, it is easy to see that there exists an isomorphism 
  $W(\delta)\isom W(A(\tilde{\delta}))$ for any unitary homomorphism $\delta:K^{\times}\rightarrow A^{\times}$.
 Moreover, it is easy to see that for any continuous homomorphisms $\delta_1,\delta_2: K^{\times}
 \rightarrow A^{\times}$ we have isomorphisms $W(\delta_1)\otimes W(\delta_2)\isom W(\delta_1\delta_2)$ and $W(\delta_1)^{\vee}
 \isom W(\delta_1^{-1})$.

\section{Deformations of trianguline $B$-pairs}

In this section, we develop the deformation theory of $B$-pairs, in particular 
we study  trianguline deformation functors for trianguline $B$-pairs,  which 
are the generalization of the deformation theory of usual $p$-adic Galois 
representations and the generalization of the theory of trianguline deformation of 
[Bel-Ch] in the case of $K=\mathbb{Q}_p$ by using 
$(\varphi,\Gamma)$-modules. We prove the pro-representabilities and formally smoothnesses
of these deformation functors and calculate their dimensions by using the calculations of 
Galois cohomology of $B$-pairs in the last section.

\subsection{Deformations of $E$-$B$-pairs}

In this subsection, we study the usual deformation theory for $E$-$B$-pairs,, which 
are the generalization of the usual (homomorphismistic zero) deformation theory for 
$E$-representations.

\begin{defn}\label{6}
Let $A$ be an object in $\mathcal{C}_E$, $W$ be an $E$-$B$-pair.
Then we say that a pair $(W_A,\iota)$ is a deformation of $W$ over $A$ 
if $W_A$ is an $A$-$B$-pair and $\iota: W_A\otimes_AE\isom W$
is an isomorphism of $E$-$B$-pairs. Let $(W_A,\iota)$, $(W'_A,\iota')$ be 
two deformations of $W$ over $A$. Then we say that $(W_A,\iota)$ and $(W'_A,\iota')$ are 
equivalent if there exists an isomorphism $f:W_A\isom W'_A$ of $A$-$B$-pairs which 
satisfies $\iota=\iota'\circ\bar{f}$, where $\bar{f}:W_A\otimes_A E\isom W'_A\otimes_AE$ 
is the reduction of $f$.
\end{defn}
\begin{defn} \label{7}
Let $W$ be an $E$-$B$-pair. Then we define the deformation functor 
$D_W$ from the category $\mathcal{C}_E$ to the category of sets by defining, for any 
$A\in\mathcal{C}_E$, 
$D_W(A):=\{$ equivalent classes $(W_A,\iota)$ of deformations of $W$ over $A$ $\}$.
\end{defn}
We simply write $W_A$ if there is no risk of confusing about $\iota$.

Next, we prove the pro-representability 
of the functor $D_W$ under suitable conditions.
For this, we recall Schlessinger's criterion for pro-representability of functors from 
$\mathcal{C}_E$ to the category of sets. 
We call a morphism $f:A'\rightarrow A$ in $\mathcal{C}_E$ a small extension if it is surjective and 
the kernel $\mathrm{Ker}(f)=(t)$ is generated by a nonzero single element $t\in A'$ 
and $\mathrm{Ker}(f)\cdot m_{A'}=0$. $E[\varepsilon]$ is the ring of dual number over $E$,
i.e. defined by $E[\varepsilon]:=E[X]/(X^2)$.

\begin{thm}\label{8}
Let $F$ be a functor from $\mathcal{C}_E$ to the category of sets such that 
$F(E)$ is a single point. For morphisms $A'\rightarrow A$, $A"\rightarrow A$ in 
$\mathcal{C}_A$, consider the natural map 
\begin{itemize}
\item[(1)] $F(A'\times_{A}A")\rightarrow F(A')\times_{F(A)}F(A")$.
\end{itemize}
Then $F$ is pro-representable if and only if $F$ satisfies properties $(H_1)$,
$(H_2)$, $(H_3)$, $(H_4)$ below:
\begin{itemize}
\item[$(H_1)$] $(1)$ is surjective if $A"\rightarrow A$ is surjective.
\item[$(H_2)$] $(1)$ is bijective when $A=E$ and $A"=E[\varepsilon]$.
\item[$(H_3)$] $\mathrm{dim}_E(t_F)<\infty$ $($ where $t_F:=F(E[\varepsilon])$ and, under 
the condition $(H_2)$, it is known that $t_F$ has a natural $E$-vector space structure$)$.
\item[$(H_4)$] $(1)$ is  bijective if $A'=A"$ and $A'\rightarrow A$ is a small extension.
\end{itemize}
\end{thm}
\begin{proof}
See [Sch68] or $\S 18$ of [Ma97].
\end{proof}

By using this criterion, we prove the pro-representability of 
$D_W$.

\begin{prop}\label{9}
Let $W$ be an $E$-$B$-pair.
If $\mathrm{End}_{G_K}(W)=E$, i.e. $W$ has only trivial morphism 
of  $E$-$B$-pairs, then $D_W$ is pro-representable by a complete 
noetherian local $E$-algebra $R_W$ with residue field $E$.
\end{prop}

For proving this proposition, we first prove some lemmas.

\begin{lemma}\label{10}
Let $\mathrm{ad}(W):=\mathrm{Hom}(W,W)$ be the internal endomorphism of $W$.
Then there exists an isomorphism of $E$-vector spaces
$D_W(E[\varepsilon])\isom \mathrm{H}^1(G_K,\mathrm{ad}(W))$.
\end{lemma}
\begin{proof}
Let $W_{E[\varepsilon]}:=(W_{E[\varepsilon],e}, W^+_{E[\varepsilon],\mathrm{dR}})$ 
be a deformation of $W$ over $E[\varepsilon]$. From this, we define an element 
in $\mathrm{H}^1(G_K,\mathrm{ad}(W))$ as follows. Because $\varepsilon W_{E[\varepsilon],e}\isom 
W_e$ and $W_{E[\varepsilon],e}/\varepsilon W_{E[\varepsilon],e} \isom W_e$ (where we put $W:=(W_e,W^+_{\mathrm{dR}})$),
we have a natural exact sequence of $B_e\otimes_{\mathbb{Q}_p}E[G_K]$-modules
\begin{equation*}
0\rightarrow W_e\rightarrow W_{E[\varepsilon],e}\rightarrow W_e\rightarrow 0
\end{equation*}
We fix an isomorphism of $B_e\otimes_{\mathbb{Q}_p}E$-modules 
$W_{E[\varepsilon],e}\isom W_e e_1\oplus W_ee_2$ such that first factor $W_e e_1$is equal to $\varepsilon W_{E[\varepsilon]}$ 
as $B_e\otimes_{\mathbb{Q}_p}E[G_K]$-module and that the above natural projection maps the second factor $W_e e_2$ to $W_e$ by $xe_2\mapsto x$ for any $x\in W_e$.  Then we define a continuous one cocycle $c_e: G_K\rightarrow \mathrm{Hom}_{B_e\otimes_{\mathbb{Q}_p}E}(W_e,W_e)$
 by $g(ye_2):=c_e(g)(gy)e_1+gye_2$ for  any $g\in G_K$ and $y\in W_e$.  
 For $W^+_{\mathrm{dR}}$,  we fix an isomorphism $W^+_{E[\varepsilon],\mathrm{dR}}\isom W^+_{\mathrm{dR}}e_1\oplus 
 W^+_{\mathrm{dR}}e'_2$ as in the case of $W_e$, then we define a one cocycle $c_{\mathrm{dR}}:G_K\rightarrow \mathrm{Hom}_{B^+_{\mathrm{dR}}\otimes_{\mathbb{Q}_p}E}(W^+_{\mathrm{dR}},W^+_{\mathrm{dR}})$ by  $g(ye'_2):=c_{\mathrm{dR}}(g)(gy)e_1+gye'_2$ for 
 any $g\in G_K$ and $y\in W^+_{\mathrm{dR}}$. Next, we define an element $c\in\mathrm{Hom}_{B_{\mathrm{dR}}\otimes_{\mathbb{Q}_p}E}
 (W_{\mathrm{dR}},W_{\mathrm{dR}})$ as follows. By tensoring $W_{E[\varepsilon],e}$ and $W^+_{E[\varepsilon],\mathrm{dR}}$ with 
 $B_{\mathrm{dR}}$ over $B_e$ or $B^+_{\mathrm{dR}}$, we have 
 an isomorphism $f:W_{\mathrm{dR}}e_1\oplus W_{\mathrm{dR}}e_2\isom W_{E[\varepsilon],\mathrm{dR}}\isom 
 W_{\mathrm{dR}}e_1\oplus W_{\mathrm{dR}}e'_2$ of $B_{\mathrm{dR}}\otimes_{\mathbb{Q}_p}E$-modules. 
 We define $c:W_{\mathrm{dR}}\rightarrow W_{\mathrm{dR}}$ by $f(ye_2):=c(y)e_1+ye'_2$ for any $y\in W_{\mathrm{dR}}$.
 Then, by definition, the triple ($c_e$, $c_{\mathrm{dR}}$, $c$) satisfies $c_e(g)-c_{\mathrm{dR}}(g)=gc-c$ in $\mathrm{Hom}_{B_{\mathrm{dR}}
 \otimes_{\mathbb{Q}_p}E}(W_{\mathrm{dR}},W_{\mathrm{dR}})$ for any $g\in G_K$, i.e. the triple $(c_e,c_{\mathrm{dR}},c)\in 
 \mathrm{H}^1(G_K,\mathrm{ad}(W))$ by the definition of Galois cohomology of $B$-pairs ($\S$ 2.1 of [Na09]). Then it is standard to check that this definition is independent of 
 the choice of fixed isomorphism $W_{E[\varepsilon],e}\isom W_ee_1\oplus W_ee_2$, etc, and it is easy to check that 
 this map defines an isomorphism $D_W(E[\varepsilon])\isom \mathrm{H}^1(G_K,\mathrm{ad}(W))$.

\end{proof}

\begin{lemma}\label{11}
Let $W_A$ be a deformation of $W$ over $A$.
If $\mathrm{End}_{G_K}(W)=E$, then $\mathrm{End}_{G_K}(W_A)=A$.
\end{lemma}

\begin{proof}
We prove this lemma by induction on the length of $A$.
When $A=E$, this is trivial. We assume that this lemma is proved for length $n$ 
rings and assume $A$ is length $n+1$. We take a small extension 
$f:A\rightarrow A'$. Because $\mathrm{End}_{G_K}(W)=\mathrm{H}^0(G_K, 
W^{\vee}\otimes W)$, then we have the following short exact sequence,
\begin{equation*}
0\rightarrow \mathrm{Ker(f)}\otimes_E\mathrm{End}_{G_K}(W)
\rightarrow \mathrm{End}_{G_K}(W_A)\rightarrow 
\mathrm{End}_{G_K}(W_A\otimes_A A').
\end{equation*}
From this and induction hypothesis, we have 
$\mathrm{length}(\mathrm{End}_{G_K}(W_A))\leqq 
\mathrm{length}(\mathrm{End}_{G_K}\allowbreak (W_A\otimes_A A'))+
\mathrm{length}(\mathrm{Ker}(f)\otimes_E \mathrm{End}_{G_K}(W))
=\mathrm{length}(A')+1=\mathrm{length}(A)$. On the other hand, 
we have a natural inclusion $A\subseteq \mathrm{End}_{G_K}(W_A)$.
So, by comparing length, we have $A=\mathrm{End}_{G_K}(W_A)$.
\end{proof}

\begin{proof} (of proposition)
Let $W$ be a rank $n$ $E$-$B$-pair satisfying $\mathrm{End}_{G_K}(W)=E$. 
For this $W$, we check the conditions $(H_i)$ of Schlessinger's criterion.
First, by Lemma $\ref{10}$, we have $\mathrm{dim}_E(D_{W}(E[\varepsilon]))=
\mathrm{dim}_E(\mathrm{H}^1(G_K,\mathrm{ad}(W)))<\infty$, so $(H_3)$ is satisfied.
Next we check $(H_1)$. Let  $f:A'\rightarrow A$, $g:A"\rightarrow A$ be morphisms in 
$\mathcal{C}_E$ such that $g$ is surjective. Let $([W_{A'}],[W_{A"}])$ be an element in 
$D_W(A')\times_{D_W(A)}D_W(A")$. We take deformations $W_{A'}:=(W_{A' e},W^{ +}_{A' \mathrm{dR}})
$, $W_{A"}:=(W_{A" e},W^{+}_{A" \mathrm{dR}})$ over $A'$ and $A"$ contained in  equivalent classes 
$[W_{A'}]$ and $[W_{A"}]$ respectively. Then we have an isomorphism $h:W_{A'}\otimes_{A'}A\isom W_{A"}\otimes_{A"}A=:W_A
:=(W_{A e}, W^+_{A\mathrm{dR}})$ 
which defines an equivalent class in $D_W(A)$. We fix a basis $e_1,\cdots,e_n$ of $W_{A' e}$ as a $B_e\otimes_{\mathbb{Q}_p}A'$
-module and write $\bar{e}_1,\cdots,\bar{e}_n$ the basis of $W_{A' e}\otimes_{A'}A$ induced from $e_1,\cdots,e_n$.
Then, by the surjection of $g:A"\rightarrow A$ and by $A"$-flatness of $W_{A" e}$, we can take a basis 
$\tilde{e}_1,\cdots,\tilde{e}_n$ of $W_{A" e}$ such that the basis 
$\bar{\tilde{e}}_1,\cdots,\bar{\tilde{e}}_n$ of $W_{A" e}\otimes_{A"}A$ induced from $\tilde{e}_1,\cdots,\tilde{e}_n$ 
satisfies $h(\bar{e}_i)=\bar{\tilde{e}}_i$ for any $i$.  Then $W^{'"}_e:=W_{A' e}\times_{W_{A e}}W_{A" e}
:=\{(x,y)\in W_{A' e}\times W_{A" e}| h(\bar{x})=\bar{y}\}$ is a free $B_e\otimes_{\mathbb{Q}_p}(A'\times_A A")$-module
 with basis $(e_1,\tilde{e}_1), \cdots, (e_n,\tilde{e}_n)$. In the same way, we can define 
 $W^{'"+}_{\mathrm{dR}}:=W^{+}_{A' \mathrm{dR}}\times_{W^+_{A \mathrm{dR}}} W^{+}_{A" \mathrm{dR}}$,  
 which is a free $B^{+}_{\mathrm{dR}}\otimes_{\mathbb{Q}_p}(A'\times_A A")$-module. Then  
 $W_{A'}\times_{W_A} W_{A"}:=(W^{'"}_e,W^{'" +}_{\mathrm{dR}})$ is a $(A'\times_A A")$-$B$-pair which is a deformation 
 of $W$ over $A'\times_A A"$ such that the equivalent class $[W_{A'}\times_{W_A}W_{A"}]\in D_W(A'\times_A A")$ maps 
 $([W_{A'}], [W_{A"}])\in D_W(A')\times_{D_W(A)}D_W(A")$. So we have proved $(H_1)$.
 
 Finally, we prove the following. If $g:A"\rightarrow A$ is surjective, then the natural map 
 $D_W(A'\times_AA")\rightarrow D_W(A')\times_{D_W(A)}D_W(A")$ is bijective, which 
 proves the condition $(H_2)$ and $(H_4)$, so we can prove the pro-representability of $D_W$.
 Let $W^{'"}_1$, $W^{'"}_2$ be deformations of $W$ over $A'\times_A A"$ such 
 that $[W^{'"}_1\otimes_{A'\times_AA"}A']=[W^{'"}_2\otimes_{A'\times_{A}A"}A']$ in $D_W(A')$ and 
 $[W^{'"}_1\otimes_{A'\times_AA"}A"]=[W^{'"}_2\otimes_{A'\times_{A}A"}A"]$ in $D_W(A")$.
 Then we want to show $[W^{'"}_1]=[W^{'"}_2]$ in $D_W(A'\times_A A")$. 
 We put $W_{1 A'}:=W^{'"}_1\otimes_{A'\times_A A"}A'$, $W_{1 A"}:=W_1^{'"}\otimes_{A'\times_A A"}A"$
 , $W_{1 A}:=W_1^{'"}\otimes_{A'\times_A A"} A$, and same for $W_{2 A'}$, $W_{2 A"}$, $W_{2 A}$.
 Then we have natural isomorphisms $W^{'"}_1\isom W_{1 A'}\times_{W_{1 A}} W_{1 A"}$ and $W_2^{'"}
 \isom W_{2 A'}\times_{W_{2 A}} W_{2 A"}$ defined as in the last paragraph. Because 
 $[W_{1 A'}]=[W_{2 A'}]$ and $[W_{1 A"}]=[W_{2 A"}]$, we have isomorphisms 
 $h':W_{1 A'}\isom W_{2 A'}$ and $h":W_{1 A"}\isom W_{2 A"}$. By reduction of these isomorphisms, 
 we get an automorphism $\bar{h}'\circ \bar{h}^{" -1}:W_{2 A}\isom W_{1 A}\isom W_{2 A}$. 
 Then, by Lemma $\ref{11}$ and the surjection of $g:A^{" \times}\rightarrow A^{\times}$, we can find an automorphism 
 $\tilde{h}:W_{2 A"}\isom W_{2 A"}$ such that $\bar{\tilde{h}}=\bar{h}'\circ \bar{h}^{" -1}$. 
 Then, ifwe define a morphism $h^{'"}:W_{1 A'}\times_{W_{1 A}} W_{1 A"}\rightarrow 
 W_{2 A"}\times_{W_{2 A}} W_{2 A'}: (x,y)\mapsto (h_1(x), \tilde{h}\circ h_2(y))$, we can see that this is 
 well-defined and is isomorphism. So we finish to prove this proposition.
 \end{proof}
 
 \begin{prop}\label{12}
 Let $W:=(W_e, W^+_{\mathrm{dR}})$ be an $E$-$B$-pair of rank $n$.
 If $\mathrm{H}^2(G_K, \mathrm{ad}(W))\allowbreak =0$, then the functor 
 $D_W$ is formally smooth.
 \end{prop}
 \begin{proof}
 Let $A'\rightarrow A$ be a small extension in $\mathcal{C}_E$,  
 we denote the kernel by $I\subseteq A'$.
 Let $W_A:=(W_{e,A}, W^+_{\mathrm{dR},A})$ be a deformation of $W$ 
 over $A$. Then it suffices to show that there exists an $A'$-$B$-pair $W_{A'}$ 
 such that $W_{A'}\otimes_{A'}A\isom W_A$. 
 We fix a basis of $W_{e,A}$ as a $B_e\otimes_{\mathbb{Q}_p}A$-module.
 By this choice, from the $G_K$-action on $W_{e,A}$, 
 we get a continuous one cocycle $\rho_e:G_K\rightarrow \mathrm{GL}_n(B_e\otimes_{\mathbb{Q}_p}A)$.
  In the same way, if we fix a basis of $W^+_{\mathrm{dR},A}$ as a $B^+_{\mathrm{dR}}\otimes_{\mathbb{Q}_p}A$-module, 
  we get a continuous cocycle $\rho_{\mathrm{dR}}:G_K\rightarrow \mathrm{GL}_n(B^+_{\mathrm{dR}}\otimes_{\mathbb{Q}_p}A)$.
   From the canonical isomorphism $W_{e,A}\otimes_{B_e}B_{\mathrm{dR}}\isom W^+_{\mathrm{dR},A}
   \otimes_{B^+_{\mathrm{dR}}}
   B_{\mathrm{dR}}$, we get a matrix $P\in\mathrm{GL}_n(B_{\mathrm{dR}}\otimes_{\mathbb{Q}_p}A)$ such that 
   $P\rho_e(g)g(P)^{-1}=\rho_{\mathrm{dR}}(g)$ for any $g\in G_K$. 
   We fix an $E$-linear section $s:A\rightarrow A'$ of $A'\rightarrow A$ and fix a lifting 
   $\tilde{P}\in \mathrm{GL}_n(B_{\mathrm{dR}}\otimes_{\mathbb{Q}_p}A')$ of $P$. From this, 
   we get  continuous liftings $\tilde{\rho}_e:=s\circ \rho_e:G_K\rightarrow \mathrm{GL}_n(B_e\otimes_{\mathbb{Q}_p}A')$ of $\rho_e$ 
   and $\tilde{\rho}_{\mathrm{dR}}:=s\circ \rho_{\mathrm{dR}}:G_K\rightarrow \mathrm{GL}_n(B^+_{\mathrm{dR}}\otimes_{\mathbb{Q}_p}A')$ of 
   $\rho_{\mathrm{dR}}$.
   By using these liftings, we define $c_e:G_K\times G_K\rightarrow I\otimes_E \mathrm{Hom}_{B_e\otimes_{\mathbb{Q}_p}E}(W_e, W_e)$ 
   by $\tilde{\rho}_e(g_1g_2)g_1(\tilde{\rho}_e(g_2))^{-1}\tilde{\rho}_e(g_1)^{-1}=1+c_{e}(g_1,g_2)\in 
   1+I \otimes_{A'}\mathrm{M}_n(B_e\otimes_{\mathbb{Q}_p}A')=1+I\otimes_{E}\mathrm{Hom}_{B_e\otimes_{\mathbb{Q}_p}E}
   (W_e,W_e)$ for any $g_1,g_2\in G_K$. In the same way, define $c_{\mathrm{dR}}:G_K\times G_K \rightarrow I\otimes_E\mathrm{Hom}_{B^+_{\mathrm{dR}}\otimes_{\mathbb{Q}_p}E}
   (W^+_{\mathrm{dR}}, W^+_{\mathrm{dR}})$ by $\tilde{\rho}_{\mathrm{dR}}(g_1g_2)g_1(\tilde{\rho}_{\mathrm{dR}}(g_2))^{-1}\tilde{\rho}_{\mathrm{dR}}(g_1)^{-1}
   =1+c_{\mathrm{dR}}(g_1,g_2)$.  We define $c:G_K\rightarrow I\otimes_E\mathrm{Hom}_{B_{\mathrm{dR}}\otimes_{\mathbb{Q}_p}E}(W_{\mathrm{dR}}, W_{\mathrm{dR}})$ 
   such that $\tilde{P}\tilde{\rho}_e(g)g(\tilde{P})^{-1}\tilde{\rho}_{\mathrm{dR}}(g)^{-1}=1 + c(g)\in 1+I\otimes_E\mathrm{Hom}_{B_{\mathrm{dR}}\otimes_{\mathbb{Q}_p}E}(W_{\mathrm{dR}}, W_{\mathrm{dR}})$. Then $c_e$ and $c_{\mathrm{dR}}$ are continuous two cocycles, i.e. satisfy $g_1c_{\ast}(g_2, g_3)-c_{\ast}(g_1g_2,g_3)+c_{\ast}(g_1, g_2g_3)
   -c_{\ast}(g_1,g_2)=0$ for any $g_1,g_2,g_3\in G_K$ ($\ast=e, \mathrm{dR}$). Moreover, we can check that $c_e$ and $c_{\mathrm{dR}}$ and $c$ satisfy 
   $c_e(g_1,g_2)-c_{\mathrm{dR}}(g_1,g_2)=g_1(c(g_2))-c(g_1g_2)+c(g_1)$ for any 
   $g_1,g_2,g_3\in G_K$ (note that the isomorphism $\mathrm{Hom}_{B_e\otimes_{\mathbb{Q}_p}E}(W_e, W_e)
   \otimes_{B_e}B_{\mathrm{dR}}\isom \mathrm{Hom}_{B^+_{\mathrm{dR}}\otimes_{\mathbb{Q}_p}E}(W^+_{\mathrm{dR}},W^+_{\mathrm{dR}})\otimes_{B^+_{\mathrm{dR}}}B_{\mathrm{dR}}$ is given by $c\mapsto \bar{P}^{-1}c\bar{P}$, where $\bar{P}\in \mathrm{GL}_n(B_{\mathrm{dR}}\otimes_{\mathbb{Q}_p}E)$
    is the reduction of $P\in \mathrm{GL}_n(B_{\mathrm{dR}}\otimes_{\mathbb{Q}_p}A)$). By definition of Galois cohomology 
    of $B$-pairs, these mean that $(c_e,c_{\mathrm{dR}}, c)$ defines an element $[(c_e,c_{\mathrm{dR}},c)]$in $I\otimes_E\mathrm{H}^2(G_K, \mathrm{ad}(W))$ . (Until here, we  don't use the condition $\mathrm{H}^2(G_K,\mathrm{ad}(W))=0$ and we can show, in the standard way, that $[(c_e,c_{\mathrm{dR}},c)]$ doesn't depend on the choice of $s$ or $\tilde{P}$, i.e. depends only on $W_A$.)
    Now, by the assumption, we have $\mathrm{H}^2(G_K,\mathrm{ad}(W))=0$, So, by definition, there exists $(f_e,f_{\mathrm{dR}}, f)
    $ such that $f_e:G_K\rightarrow I\otimes_{E}\mathrm{Hom}_{B_e\otimes_{\mathbb{Q}_p}E}(W_e, W_e)$ and 
    $f_{\mathrm{dR}}:G_K\rightarrow I\otimes_E\mathrm{Hom}_{B^+_{\mathrm{dR}}\otimes_{\mathbb{Q}_p}E}(W^+_{\mathrm{dR}}, W^+_{\mathrm{dR}})$ are continuous map and $f\in I\otimes_E\mathrm{Hom}_{B_{\mathrm{dR}}\otimes_{\mathbb{Q}_p}E}(W_{\mathrm{dR}}, W_{\mathrm{dR}})$ 
    and these satisfy $c_e(g_1,g_2)=g_1f_e(g_2)-f_e(g_1g_2)+f_e(g_1)$ and $c_{\mathrm{dR}}(g_1,g_2)=g_1f_{\mathrm{dR}}(g_2)
    -f_{\mathrm{dR}}(g_1g_2)+f_{\mathrm{dR}}(g_1)$ and $c(g_1)=f_{\mathrm{dR}}(g_1)-\bar{P}^{-1}f_{e}(g_1)\bar{P}+(g_1f-f)$ 
    for any $g_1,g_2\in G_K$.  
    Then we define new liftings $\rho'_{e}:G_K\rightarrow \mathrm{GL}_n(B_e\otimes_{\mathbb{Q}_p}A')$ by 
    $\rho'_e(g):=(1+f_e(g))\tilde{\rho}_e(g)$, $\rho'_{\mathrm{dR}}(g):G_K\rightarrow \mathrm{GL}_n(B^+_{\mathrm{dR}}\otimes_{\mathbb{Q}_p}A')$ by $\rho'_{\mathrm{dR}}(g):=(1+f_{\mathrm{dR}}(g))\tilde{\rho}_{\mathrm{dR}}(g)$ and 
    $P':=(1+f)\tilde{P}\in \mathrm{GL}_n(B_{\mathrm{dR}}\otimes_{\mathbb{Q}_p}A')$. Then we can check 
    that these satisfy $\rho'_e(g_1g_2)=\rho'_e(g_1)\rho'_{e}(g_2)$ and $\rho'_{\mathrm{dR}}(g_1g_2)
    =\rho'_{\mathrm{dR}}(g_1)\rho'_{\mathrm{dR}}(g_2)$ and $P'\rho'_e(g_1)g(P')^{-1}=\rho'_{\mathrm{dR}}(g_1)$ for 
    any $g_1,g_2\in G_K$.  By definition of $A'$-$B$-pair, this means that $(\rho'_e,\rho'_{\mathrm{dR}},P')$ defines 
    an $A'$-$B$-pair and ,by definition, this is a lift of $W_A$. We finish the proof of this proposition.
    \end{proof}
    
    \begin{corollary}\label{13}
    Let $W$ be an $E$-$B$-pair of rank $n$.
    If $\mathrm{End}_{G_K}(W)=E$ and $\mathrm{H}^2(G_K, \mathrm{ad}(W))=0$ 
    and if we denote $d:=[K:\mathbb{Q}_p]n^2 +1$, 
    then the functor $D_W$ is pro-representable by $R_W$ such that $R_W\isom E[[T_1,\cdots,T_d]]$.
    \end{corollary}
    \begin{proof}
    The existence and formally smoothness of $R_W$ follows from Proposition $\ref{9}$ and Proposition $\ref{12}$.
    For dimension, by Theorem $\ref{h}$ and Lemma $\ref{10}$, we have 
    $\mathrm{dim}_ED_W(E[\varepsilon])=\mathrm{dim}_E\mathrm{H}^1(G_K, \mathrm{ad}(W))=[K:\mathbb{Q}_p]n^2
    +\mathrm{dim}_E\mathrm{H}^0(G_K,\mathrm{ad}(W))+\mathrm{dim}_E\mathrm{H}^2(G_K,\mathrm{ad}(W))=[K:\mathbb{Q}_p]n^2
    +1$.
    \end{proof}
 
\subsection{Trianguline  deformations of trianguline $B$-pairs}

In this subsection, we define the trianguline deformation functor for 
split trianguline $E$-$B$-pairs and prove the pro-representability and formally smoothness 
under some conditions and calculate the dimension of the universal 
deformation ring of  this functor. In $\S$ 2 of [Bel-Ch 09], Bella\"iche-Chenevier 
proved all these in the case 
of $K=\mathbb{Q}_p$ by using $(\varphi,\Gamma)$-modules over the Robba ring and by using Colmez's 
theory of trianguline representations for $K=\mathbb{Q}_p$ ([Co08]). 
We generalize their theory by using 
$B$-pairs and the theory of trianguline representations for any $p$-adic fields ([Na09] or $\S$ 2). 

First, we define the notion of split  trianguline $A$-$B$-pairs  as follows.
\begin{defn}
Let $W$ be a rank $n$ $A$-$B$-pairs. Then we say that $W$ is a split trianguline 
$A$-$B$-pair if there exists a sequence of sub $A$-$B$-pairs 
$\mathcal{T}: 0=W_0\subseteq W_1\subseteq W_2\subseteq\cdots \subseteq W_{n-1}
\subseteq W_n=W$ such that  $W_i$ is saturated in $W_{i+1}$ and the 
quotient $W_{i+1}/W_i$ is a rank one $A$-$B$-pair for any $0\leqq i\leqq 
n-1$.

By Proposition $\ref{5}$, there exist continuous homomorphisms $\delta_i:K^{\times}\rightarrow A^{\times}$ 
such that $W_i/W_{i-1}\isom W(\delta_i)$ for any $1\leqq i\leqq n$.
We say that $\{\delta_i\}_{i=1}^n$ is the parameter of triangulation 
$\mathcal{T}$.
\end{defn}

Next, we define the trianguline deformation functor. 
Let $W$ be a rank $n$ split trianguline $E$-$B$-pair. We fix a triangulation 
$\mathcal{T}: 0\subseteq W_1\subseteq\cdots\subseteq W_{n-1}\subseteq 
W_n=W$ of $W$. Under this condition, we define the trianguline deformation 
as follows.

\begin{defn}
Let $A$ be an object in $\mathcal{C}_E$.
We say that $(W_A,\iota, \mathcal{T}_A)$ is a trianguline deformation of $(W,\mathcal{T})$ over $A$
if $(W_A,\iota)$ is a deformation of $W$ over $A$ and $W_A$ is a split trianguline $A$-$B$-pair with 
a triangulation $\mathcal{T}_A:0\subseteq W_{1,A}\subseteq \cdots \subseteq W_{n,A}=W_A$ 
 such that  $\iota(W_{i,A}\otimes_A E)=W_i$ for any $1\leqq i\leqq n$.
Let $(W_A,\iota,\mathcal{T}_A)$ and $(W'_A,\iota',\mathcal{T}'_A)$ be two trianguline 
deformations of $(W,\mathcal{T})$ over $A$. Then we say that $(W_A,\iota,\mathcal{T}_A)$ and 
$(W'_A,\iota',\mathcal{T}'_A)$ are equivalent if there exists an isomorphism of $A$-$B$-pairs 
$f:W_A\isom W'_A$ satisfying $\iota=\iota'\circ \bar{f}$ and $f(W_{i,A})=W'_{i,A}$ for 
any $1\leqq i\leqq n$.
\end{defn}
\begin{defn}
Let $W$ be a split trianguline $E$-$B$-pair with a triangulation $\mathcal{T}$. 
Then we define the trianguline deformation functor $D_{W,\mathcal{T}}$ 
from the category $\mathcal{C}_E$ to the category of sets by defining, for any $A\in\mathcal{C}_E$, 
$D_{W,\mathcal{T}}(A):=\{$ equivalent classes $(W_A,\iota,\mathcal{T}_A)$ of trianguline 
deformation of $(W,\mathcal{T})$ over $A$ $\}$.
\end{defn}

By definition, we have a natural map of functors from $D_{W,\mathcal{T}}$ to $D_{W}$ by forgetting 
triangulation, i.e. by defining $D_{W,\mathcal{T}}(A)\rightarrow D_{W}(A): [(W_A,\iota,\mathcal{T}_A)]
\mapsto [(W_A,\iota)]$. In general, by this map, $D_{W,\mathcal{T}}$ is not a sub functor of 
$D_W$, i.e. a deformation $W_A$ can have many liftings of triangulation of $\mathcal{T}$.
So we give a sufficient condition for $D_{W,\mathcal{T}}$ to be a subfunctor of $D_W$.
Let $\{\delta_i\}_{i=1}^n$ be the parameter of triangulation $\mathcal{T}$.
\begin{lemma}\label{14}
If, for any $1\leqq i<j\leqq n$, we have $\delta_j/\delta_i\not= \prod_{\sigma\in\mathcal{P}}\sigma(x)^{k_{\sigma}}$ for 
any $\{k_{\sigma}\}_{\sigma\in\mathcal{P}}\in \prod_{\sigma\in\mathcal{P}} \mathbb{Z}_{\leqq 0}$, then 
the functor $D_{W,\mathcal{T}}$ is a sub functor of $D_W$.
\end{lemma}
\begin{proof}
Let $W_A$ be a deformation of $W$ over $A$, let $0\subseteq W_{A ,1}\subseteq \cdots \subseteq W_{A ,n-1}\subseteq W_{A}$ 
and $0\subseteq W'_{A ,1}\subseteq \cdots \subseteq W'_{A, n-1}\subseteq W_{A}$ be two triangulations which are 
lifts of $\mathcal{T}$, then it suffices to show that $W_{A, i}=W'_{A, i}$ for any $i$. Moreover, by induction, it suffices 
to show $W_{A ,1}= W'_{A ,1}$. For proving this, first we consider $\mathrm{Hom}_{G_K}(W_{1,A}, W_A)$. 
This is equal to  $\mathrm{H}^0(G_K, W^{\vee}_{1, A}\otimes W_A)$. Because $\mathrm{H}^0(G_K, -)$ is left exact 
and $\mathrm{H}^0(G_K, W(\delta))=0$ for any $\delta:K^{\times}\rightarrow E^{\times}$ such that 
$\delta\not= \prod_{\sigma\in \mathcal{P}}\sigma(x)^{k_{\sigma}}$ for any $\{k_{\sigma}\}_{\sigma\in\mathcal{P}}
\in \prod_{\sigma\in\mathcal{P}}\mathbb{Z}_{\leqq 0}$ by Proposition $\ref{i}$, we have $\mathrm{H}^0(G_K, W_{1, A}^{\vee}\otimes 
(W_{i+1, A}/W_{i, A}))\allowbreak =\mathrm{H}^0(G_K, W_{1, A}^{\vee}\otimes( W'_{i+1, A}/W'_{i, A}))=0$ for any $i\geqq 2$. 
So we get $\mathrm{Hom}_{G_K}(W_{1, A}, W_{1,A})\allowbreak =\mathrm{Hom}_{G_K}(W_{1,A}, \allowbreak W_A)=\mathrm{Hom}_{G_K}(W_{1, A}, 
W'_{1,A})$. This means that the given inclusion $W_{1, A}\hookrightarrow W_A$ factors through $W'_{1, A}\hookrightarrow W_A$. 
By symmetry, the inclusion $W'_{1,A}\hookrightarrow W_A$ factors through $W_{1, A}\hookrightarrow W_{A}$. 
So we get $W_{1,A}=W'_{1, A'}$.
\end{proof}

\begin{prop}\label{15}
Let $W$ be a trianguline representations with a triangulation $\mathcal{T}$ 
such that the parameter $\{\delta_i\}_{\i=1}^n$ satisfies, for any $1\leqq i<j\leqq n$, $\delta_j/\delta_i\not=\prod_{\sigma\in \mathcal{P}}
\sigma(x)^{k_{\sigma}}$ for any $\{k_{\sigma}\}_{\sigma\in\mathcal{P}}\in \prod_{\sigma\in\mathcal{P}}\mathbb{Z}_{\leqq 0}$.
Then the natural map of functors $D_{W,\mathcal{T}}\rightarrow D_{W}$ is relatively representable.
\end{prop}
\begin{proof}
By $\S 23$ of [Ma97], it suffices to check that the map $D_{W,\mathcal{T}}\rightarrow D_W$  satisfies the fallowing three conditions (1), (2), (3).
 (1) :For any map $A\rightarrow A'$ in $\mathcal{C}_E$ and $W_{A}\in D_{W,\mathcal{T}}(A)$, then 
 $W_A\otimes_A A'\in D_{W,\mathcal{T}}(A')$. (2): For any maps $A'\rightarrow A$ and $A"\rightarrow A$ in $\mathcal{C}_E$
  and $W^{'"}\in D_W(A'\times_A A")$, if $W^{'"}\otimes_{A'\times_A A"} A'\in D_{W,\mathcal{T}}(A')$ and 
  $W^{'"}\otimes_{A'\times_A A"}A"\in D_{W,\mathcal{T}}(A")$, then $W^{'"}\in D_{W,\mathcal{T}}(A'\times_A A")$.
   (3): For any inclusion $A\hookrightarrow A'$ in $\mathcal{C}_E$ and $W_{A}\in D_W(A)$, if $W_A\otimes_A A'\in 
   D_{W,\mathcal{T}}(A')$, then $W_A\in D_{W,\mathcal{T}}(A)$. 
   The condition (1) is trivial. For (2), let $W^{'"}\in D_{W}(A'\times_A A")$ be such that $W_{A'}:=W^{'"}\otimes_{A'\times_A A"}A'
   \in D_{W,\mathcal{T}}(A')$ and $W_{A"}:=W^{'"}\otimes_{A'\times_A A"}A"\in D_{W,\mathcal{T}}(A")$. We put $W_A:=W^{'"}\otimes_{A'\times_A A"}A$. Then as in proof of Proposition $\ref{9}$, we have $W^{'"}\isom W_{A'}\times_{W_A} W_{A"}$. 
   And, by Lemma $\ref{14}$, the triangulations of $W_A$ induced from $W_{A'}$ and $W_{A"}$ are same, so 
   these triangulations induce a triangulation of $W^{'"}\isom W_{A'}\times_{W_A} W_{A"}$, i.e. $W^{'"}\in D_{W,\mathcal{T}}(A'\times_A A")$.
   Finally, we prove the condition (3). Let $W\in D_W(A)$ and $A\hookrightarrow A'$ be an inclusion such that 
   $W_A\otimes_A A'\in D_{W,\mathcal{T}}(A')$. Let $0\subseteq W_{1,A'}\subseteq\cdots\subseteq W_{n-1,A'}\subseteq W_{A}
   \otimes_A A'$ be a triangulation lifting of $\mathcal{T}$. By induction on rank of $W$, it suffices to show that there exists 
   a rank one sub $A$-$B$-pair $W_{1,A}\subseteq W_A$ such that $W_{1,A}\otimes_{A}A'=W_{1,A'}$ and that $W_A/W_{1,A}$ is an $A$-$B$-pair.
    By Proposition $\ref{5}$, by twisting, 
   we may assume that $W_{1,A'}\isom (B_e\otimes_{\mathbb{Q}_p}A', B^+_{\mathrm{dR}}\otimes_{\mathbb{Q}_p}A')$ is 
   trivial $A'$-$B$-pair (so $\delta_1:K^{\times}\rightarrow E^{\times}$ is trivial homomorphism). 
   Then, by the same way as in the proof of Lemma $\ref{14}$, we have  $A'\isom \mathrm{H}^0(G_K, W_{1,A'})=
   \mathrm{H}^0(G_K, W_{A}\otimes_A A')$ and $E=\mathrm{H}^0(G_K, W_{1})=\mathrm{H}^0(G_K, W)$.
   Then, by using these and the left exactness of $\mathrm{H}^0(G_K,-)$, we have 
   $\mathrm{length}(A')=\mathrm{length}(\mathrm{H}^0(G_K, W_{A'}))\leqq \mathrm{length}(\mathrm{H}^0(G_K, W_A))+ 
   \mathrm{length}(\mathrm{H}^0(G_K, W_A\otimes(A'/A))\leqq \mathrm{length}(A)+\mathrm{length}(A'/A)
   =\mathrm{length}(A')$.  So we have $\mathrm{length}(\mathrm{H}^0(G_K, W_A))=\mathrm{length}(A)$. 
   From this, by similar argument, we can show the natural sequence 
   $0\rightarrow \mathrm{H}^0(G_K, W_A\otimes_A m_A)\rightarrow \mathrm{H}^0(G_K, W_A)
   \rightarrow \mathrm{H}^0(G_K, W)\rightarrow 0$ is exact. 
   We take a lift $v\in \mathrm{H}^0(G_K, W_A)$ which maps to $1\in E\isom \mathrm{H}^0 (G_K, W)$ 
   by the surjection $\mathrm{H}^0(G_K, W_A)\rightarrow \mathrm{H}^0(G_K, W)$. 
   Then we claim that $A\isom Av=\mathrm{H}^0(G_K, W_A)$. Because the surjection $\mathrm{H}^0(G_K, W_A)\rightarrow \mathrm{H}^0(G_K, W)$ factors through 
   $\mathrm{H}^0(G_K, W_A)\hookrightarrow \mathrm{H}^0(G_K, W_A\otimes_A A')\rightarrow 
   \mathrm{H}^0(G_K, W)$ and $\mathrm{H}^0(G_K, W_{A}\otimes_A A')\isom A'$, the image of $v$ in $\mathrm{H}^0(G_K, 
   W_A\otimes_A A')$ must be a generator of $\mathrm{H}^0(G_K, W_A\otimes_A A')$. So we have $A\isom Av\subseteq 
   \mathrm{H}^0(G_K, W_A)$. By comparing length, we get the claim. By using this $v\in \mathrm{H}^0(G_K, W_A)
   =\mathrm{Hom}_{G_K}(B_A, W_A)$ (here $B_A$ is the trivial $A$-$B$-pair), we define a map $v:B_A\rightarrow W_A$.
   Then $v\otimes \mathrm{id}_{A'}:B_A'\rightarrow W_A'$ induces an isomorphism $B_A'\isom W_{1,A'}$ because 
   $\mathrm{H}^0(G_K, W_{1,A'})=\mathrm{H}^0(G_K, W_A\otimes_A A')$.  From this, $v$ is injective and if we put $W_{1,A}
   \subseteq W_A$ the image of $v$, then $W_{1,A}\otimes_A A'=W_{1,A'}$. Finally, we need to prove 
   that the quotient $W_A/W_{1,A}$ is an $A$-$B$-pair. First, $W_A/W_{1,A}$ is an $E$-$B$-pair because 
   $W_{1,A}$ is saturated in $W_{A}\otimes_A A'$ as an $E$-$B$-pair, so also saturated in $W_A$. For proving $A$-flatness 
   of $W_{A}/W_{1,A}:=(W_{0, e}, W^+_{0,\mathrm{dR}})$, by definition of $v$, 
   we have that the natural map $W_{1,A}\otimes_A E\hookrightarrow W_{A}\otimes_A E$ is injective. From this, we get
   $A$-flatness of $W_{0,e}$ and $W^+_{0,\mathrm{dR}}/tW^+_{0,\mathrm{dR}}$, so 
   $W_A/W_{1,A}$ is an $A$-$B$-pair. We finish to prove this proposition.
     \end{proof}
     \begin{corollary}\label{16}
     Let $W$ be a trianguline $E$-$B$-pair with a triangulation $\mathcal{T}$ such that $\mathrm{End}_{G_K}(W)=E$
      and the parameter $\{\delta_i\}_{i=1}^n$ of $\mathcal{T}$ satisfies, for any $1\leqq i<j\leqq n$, 
      $\delta_j/\delta_i\not= \prod_{\sigma\in \mathcal{P}}\sigma(x)^{k_{\sigma}}$ for any $\{k_{\sigma}\}_{\sigma\in \mathcal{P}}
      \in \prod_{\sigma\in\mathcal{P}}\mathbb{Z}_{\leqq 0}$. 
      Then the functor $D_{W,\mathcal{T}}$ is pro-representable by a quotient $R_{W,\mathcal{T}}$ of $R_W$.
      \end{corollary}
      \begin{proof} This follows from Proposition $\ref{9}$ and Proposition $\ref{15}$.
      \end{proof}
      
 Next, we prove  the formally smoothness of the functor $D_{W,\mathcal{T}}$.
 \begin{prop}\label{17}
 Let $W$ be a trianguline $E$-$B$-pair of rank $n$ with triangulation $\mathcal{T}$ such that
  whose parameter $\{\delta_i\}_{i=1}^{n}$ satisfies, for any $1\leqq i< j\leqq n$, $\delta_i/\delta_j\not=
  |\mathrm{N}_{K/\mathbb{Q}_p}(x)|\prod_{\sigma\in\mathcal{P}}\sigma(x)^{k_{\sigma}} $ for any 
  $\{k_{\sigma}\}_{\sigma\in\mathcal{P}}\in \prod_{\sigma\in\mathcal{P}}\mathbb{Z}_{\geqq 1}$.
  Then the functor $D_{W,\mathcal{T}}$ is formally smooth.
  \end{prop}
  \begin{proof}
  We prove this proposition by induction of rank of $W$. 
  When $W$ is rank one, then we have $D_{W,\mathcal{T}}=D_W$ and 
 $\mathrm{ad}(W)=B_E$ is the trivial $E$-$B$-pair. 
 So $\mathrm{H}^2(G_K,\mathrm{ad}(W))=0$ by Proposition $\ref{i}$. So, 
 by Proposition $\ref{12}$, $D_{W,\mathcal{T}}$ is formally smooth in this case.
 Let's assume that the proposition is proved for rank $n-1$ $E$-$B$-pairs,
  let $W$ be a rank $n$ $E$-$B$-pair with triangulation 
  $\mathcal{T}:0\subseteq W_1\subseteq \cdots \subseteq W_{n-1}\subseteq W_n=W$ 
  whose parameter $\{\delta_i\}_{i=1}^n$ satisfying the condition as above. 
  Let $A'\rightarrow A$ be a small extension in $\mathcal{C}_E$ and let 
  $W_A$ be a trianguline deformation of $(W,\mathcal{T})$ with triangulation 
  $\mathcal{T}_A:0\subseteq W_{1,A}\subseteq \cdots \subseteq W_{n-1,A}\subseteq W_{n,A}=W_A$ 
  which is a lift of $\mathcal{T}$. It suffices to prove that there exists a split trianguline $A'$-$B$-pair 
  $W_{A'}$ with a triangulation $0\subseteq W_{1,A'}\subseteq \cdots \subseteq W_{n-1.A'}\subseteq 
  W_{n,A'}=W_{A'}$ lifting of $W_A$ and $\mathcal{T}_A$. We take a lifting as follows. 
  First, because $W_{n-1}$ is a trianguline $E$-$B$-pair of rank $n-1$ satisfying the conditions as above, 
  so by induction hypothesis, there exists a rank $n-1$ trianguline $A'$-$B$-pair $W_{n-1,A'}$ 
  with a triangulation $0\subseteq W_{1,A'}\subseteq \cdots \subseteq W_{n-2,A'}\subseteq W_{n-1,A'}$ 
  which is a lift of $W_{n-1,A}$ and $0\subseteq W_{1,A}\subseteq\cdots\subseteq W_{n-1,A}$. 
  If we put $\mathrm{gr}_nW_A:=W_{A}/W_{n-1,A}$, then, by rank one case and by Proposition $\ref{5}$, there exists 
  a continuous homomorphism $\delta_{n,A'}:K^{\times}\rightarrow A^{' \times}$ such that the rank one $A'$-$B$-pair 
  $W(\delta_{n,A'})$ satisfies $W(\delta_{n,A'})\otimes_{A'}A=W(\delta_{n,A})\isom \mathrm{gr}_nW_A$.
  (where $\delta_{n,A}:K^{\times}\rightarrow A^{\times}$ is the reduction of $\delta_{n,A'}$.) 
  Then $[W_A]$ is an element in $\mathrm{Ext}^1(W(\delta_{n,A}), W_{n-1,A})\isom \mathrm{H}^1(G_K,
  W_{n-1,A}(\delta_{n,A}^{-1}))$. If we take the long exact sequence associated 
  to 
   \begin{equation*}
  0\rightarrow I\otimes_E W_{n-1}(\delta_n^{-1})\rightarrow W_{n-1,A'}(\delta^{-1}_{n,A'})\rightarrow 
  W_{n-1,A}(\delta^{-1}_{n,A})\rightarrow 0
  \end{equation*}
  
   (where $I\subseteq A'$ is the kernel of $A'\rightarrow A$),  then we get a long exact sequence
  \[
  \begin{array}{ll}
  \cdots \rightarrow \mathrm{H}^1(G_K,W_{n-1,A'}(\delta^{-1}_{n,A'}))& \rightarrow \mathrm{H}^1(G_K, W_{n-1,A}(\delta^{-1}_{n,A}))\\
                                        &\rightarrow I\otimes_E \mathrm{H}^2(G_K, W_{n-1}(\delta^{-1}_n))\rightarrow\cdots 
  \end{array}
  \]
  By assumption on $\{\delta_i\}_{i=1}^n$ and by Proposition $\ref{i}$, we have $\mathrm{H}^2(G_K, W_{n-1}(\delta^{-1}_n))=0$. 
  So we can take a $[W_{A'}]\in\mathrm{Ext}^1(W(\delta_{n,A'}),W_{n-1,A'})\isom\mathrm{H}^1(G_K, W_{n-1,A'}(\delta^{-1}_{n,A'}))$ 
  which is a trianguline lift of $[W_A]$. This proves the proposition.
  \end{proof}
  
  Next, we calculate the dimension of $D_{W,\mathcal{T}}$. For this, we interpret 
  $D_{W,\mathcal{T}}(E[\varepsilon])$ in terms of Galois cohomology of $B$-pair as in Lemma $\ref{10}$.
  Let $W$ be a trianguline $E$-$B$-pair with triangulation $\mathcal{T}:0\subseteq W_1\subseteq \cdots\subseteq 
  W_{n-1}\subseteq W_n=W$ whose parameter is $\{\delta_i\}_{i=1}^n$. 
  Then we define an $E$-$B$-pair $\mathrm{ad}_{\mathcal{T}}(W)$ by 
  $\mathrm{ad}_{\mathcal{T}}(W):=\{f\in\mathrm{ad}(W)| f(W_i)\subseteq W_i $ for any $1\leqq i\leqq n \}$. 
  \begin{lemma}\label{18}
  Let $W$ be a trianguline $E$-$B$-pair. Then there exists 
  a canonical bijection of sets $D_{W,\mathcal{T}}(E[\varepsilon])
  \isom \mathrm{H}^1(G_K, \mathrm{ad}_{\mathcal{T}}(W))$. 
  In particular, if $D_{W,\mathcal{T}}$ has a canonical structure of $E$-vector apace 
  $($c.f.  the condition $(2)$ in Shlessinger's criterion $\ref{8}$$)$, then this bijection is an $E$-linear isomorphism.
  \end{lemma}
  \begin{proof}
The construction of the map $D_{W,\mathcal{T}}(E[\varepsilon])\rightarrow 
\mathrm{H}^1(G_K,\mathrm{ad}_{\mathcal{T}}(W))$ is same as in the proof of 
Lemma $\ref{10}$. We put $\mathrm{ad}_{\mathcal{T}}(W):=(\mathrm{ad}_{\mathcal{T}}(W_e), \mathrm{ad}_{\mathcal{T}}(W^+_{\mathrm{dR}}))$. Let $W_{E[\varepsilon]}:=(W_{e,E[\varepsilon]},W^+_{\mathrm{dR},E[\varepsilon]})$ be a trianguline deformation 
of $(W,\mathcal{T})$ over $E[\varepsilon]$ whose triangulation $\mathcal{T}_{E[\epsilon]}$ is $0\subseteq W_{1,E[\varepsilon]},\subseteq\cdots\subseteq 
W_{n-1,E[\varepsilon]}\subseteq W_{n,E[\varepsilon]}=W_{E[\varepsilon]}$. Then we can take a splitting 
$W_{e,E[\varepsilon]}=W_e e_1\oplus W_ee_2$ as a filtered $B_e\otimes_{\mathbb{Q}_p}E$-module such that 
$W_e e_1=\varepsilon W_{e,E[\varepsilon]}$ and the natural map $W_ee_2\hookrightarrow W_{e,E[\varepsilon]}
\rightarrow W_{e,E[\varepsilon]}/\varepsilon W_{e,E[\varepsilon]}\isom W_e$ is such that $ye_2\mapsto y$ for any $y\in W_e$.
Then, if we define $c_e:G_K\rightarrow \mathrm{Hom}_{B_e\otimes_{\mathbb{Q}_p}E}(W_e, W_e)$ as in the proof 
of Lemma $\ref{10}$, we can check that the image of $c_e$ is contained in $\mathrm{ad}_{\mathcal{T}}(W_e)$. In the same way,  
we can define $c_{\mathrm{dR}}:G_K\rightarrow \mathrm{ad}_{\mathcal{T}}(W^+_{\mathrm{dR}})$ from a 
filtered splitting $W^+_{\mathrm{dR},E[\varepsilon]}=W^+_{\mathrm{dR}}e_1\oplus W^+_{\mathrm{dR}}e'_2$. 
Moreover, we can define $c\in \mathrm{ad}_{\mathcal{T}}(W_{\mathrm{dR}})$ by $ye_2=c(y)e_1+ye'_2$ for any $y\in W_{\mathrm{dR}}$.
Then the map $D_{W,\mathcal{T}}(E[\varepsilon])\rightarrow \mathrm{H}^1(G_K, \mathrm{ad}_{\mathcal{T}}(W)):[(W_{E[\varepsilon]},
\mathcal{T}_{E[\varepsilon]})]\mapsto [(c_e,c_{\mathrm{dR}},c)]$ defines an bijection and, when $D_{W,\mathcal{T}}(E[\varepsilon])$ 
has a canonical $E$-vector space structure, this is an $E$-linear isomorphism.
\end{proof}

We calculate the dimension of $R_{W,\mathcal{T}}$.
\begin{prop}\label{19}
Let $W$ be a trianguline $E$-$B$-pair of rank $n$ such that  $\mathrm{End}_{G_K}(W)\allowbreak =E$ whose triangulation 
$\mathcal{T}:0\subseteq W_1\subseteq \cdots\subseteq W_{n-1}\subseteq W_n=W$ satisfies 
the following. The parameter $\{\delta_i\}_{i=1}^n$ satisfies, for any $1\leqq i < j\leqq n$,  $\delta_j/\delta_i\not= \prod_{\sigma\in\mathcal{P}}\sigma(x)^{k_{\sigma}}$ 
for any $\{k_{\sigma}\}_{\sigma\in\mathcal{P}}\in \prod_{\sigma\in\mathcal{P}}\mathbb{Z}_{\leqq 0}$ and 
$\delta_i/\delta_j\not= |\mathrm{N}_{K/\mathbb{Q}_p}(x)|\prod_{\sigma\in\mathcal{P}}\sigma(x)^{k_{\sigma}}$ 
for any $\{k_{\sigma}\}_{\sigma\in\mathcal{P}}\in \prod_{\sigma\in\mathcal{P}}\mathbb{Z}_{\geqq 1}$. 
Then the universal trianguline deformation ring $R_{W,\mathcal{T}}$ is a quotient ring of $R_W$ such 
that $R_{W,\mathcal{T}}\isom E[[T_1,\cdots, T_{d_n}]]$ where $d_n:=\frac{n(n+1)}{2}[K:\mathbb{Q}_p]+1$.
\end{prop}
\begin{proof}
From Proposition $\ref{15}$ and Proposition $\ref{17}$ and Lemma $\ref{18}$, it suffices to show that $\mathrm{dim}_E\mathrm{H}^1(G_K, \mathrm{ad}_{\mathcal{T}}(W))=d_n$. 
We prove this by induction on rank $n$ of $W$. 
When $n=1$, then $\mathrm{ad}_{\mathcal{T}}(W)=\mathrm{ad}(W)=B_E$, so the proposition 
follows from Proposition $\ref{i}$. Let $W$ be a rank $n$ trianguline $E$-$B$-pair as above, let $\mathcal{T}_{n-1}:0\subseteq 
W_{1}\subseteq \cdots \subseteq W_{n-2}\subseteq W_{n-1}$ be the triangulation of $W_{n-1}(\subseteq W)$. 
Then , by definition of $\mathrm{ad}_{\mathcal{T}}(W)$, for any $f\in \mathrm{ad}_{\mathcal{T}}(W)$, the restriction of $f$ 
to $W_{n-1}$ is an element of $\mathrm{ad}_{\mathcal{T}_{n-1}}(W_{n-1})$ and this defines a short exact sequence 
of $E$-$B$-pair,
\begin{equation*}
0\rightarrow \mathrm{Hom}(W(\delta_n), W)\rightarrow \mathrm{ad}_{\mathcal{T}}(W)\rightarrow 
\mathrm{ad}_{\mathcal{T}_{n-1}}(W_{n-1})\rightarrow 0
\end{equation*}
From this, we have $\mathrm{rank}(\mathrm{ad}_{\mathcal{T}}(W))=\mathrm{rank}(\mathrm{ad}_{\mathcal{T}_{n-1}}(W_{n-1}))
+n=1+2+\cdots +n=\frac{n(n+1)}{2}$.  From this and Theorem $\ref{h}$, 
it suffices to show that $\mathrm{H}^0(G_K,\mathrm{ad}_{\mathcal{T}}(W))=E$ and $\mathrm{H}^2(G_K,\mathrm{ad}_{\mathcal{T}}(W))=0$
. For $H^0$, this follows from  $E\subseteq \mathrm{H}^0(G_K, \mathrm{ad}_{\mathcal{T}}(W))\subseteq \mathrm{H}^0(G_K, 
\mathrm{ad}(W))=E$. We prove $\mathrm{H}^2(G_K, \mathrm{ad}_{\mathcal{T}}(W))=0$ by induction of rank of $W$. 
When $n=1$, this follows from Proposition $\ref{i}$. When $W$ is rank $n$, then from the above short exact sequence, we have the following 
 long exact sequence ,
 \begin{equation*}
 \cdots \rightarrow \mathrm{H}^2(G_K, \mathrm{Hom}(W(\delta_n), W))\rightarrow \mathrm{H}^2(G_K, \mathrm{ad}_{\mathcal{T}}(W))
 \rightarrow \mathrm{H}^2(G_K, \mathrm{ad}_{\mathcal{T}_{n-1}}(W_{n-1}))\rightarrow 0
\end{equation*}
Then, by Proposition $\ref{i}$ and by the assumption on $\{\delta_i\}_{i=1}^n$, we have $\mathrm{H}^2(G_K, \mathrm{Hom}\allowbreak (W(\delta_n), W))=0$. 
So, by induction hypothesis, we have $\mathrm{H}^2(G_K, \mathrm{ad}_{\mathcal{T}}(W))=0$. We finish the proof of this 
proposition.

\end{proof}

\section{Deformations of benign $B$-pairs}

In this final section, we study a very important class, called benign $B$-pairs, 
of trianguline $E$-$B$-pairs, which satisfies very good properties for trianguline 
deformations and plays a crucial role in Zariski density problem of modular Galois (or crystalline) 
representations in some deformation spaces of global (or local) $p$-adic representations.
This  class was defined by Kisin in the case of $K=\mathbb{Q}_p$ and rank $W$ is 2 in [Ki03] and [Ki08b].
He studied some deformation theoretic properties of this class in [Ki03] and used these in a crucial way in his 
proof of Zariski density of two dimensional crystalline representations of $G_{\mathbb{Q}_p}$.
Bella\"iche-Chenevier ([Bel-Ch09]) and Chenevier ([Ch08],[Ch09]) generalized this class for 
higher dimensional case and $K=\mathbb{Q}_p$ case, they defined and studied in detail
crystalline representations with non-critical refinements. In particular, Chenevier ([Ch08],[Ch09]) 
discovered and proved a cruicially important preoperty concerning to tangent spaces of these deformation 
rings. (This was discovered by Kisin ([Ki08b]) implicitly when $K=\mathbb{Q}_p$ and rank 2 case.) 
In fact, by using this property, Chenevier ([Ch08],[Ch09]) proved  many types of theorems concerning 
Zariski density of modular Galois 
representations in some deformation spaces of global $p$-adic representations. 

The aim of this section is to generalize this Chenevier's theorem to any $K$-case.

\subsection{Benign $B$-pairs}
Let $P(X)\in \mathcal{O}_K[X]$ be a polynomial such that $P(X)\equiv \pi_K X$ (mod deg $2$) 
and that $P(X)\equiv X^q$ (mod $\pi_K$), where $q:=p^f$ and $f:=[K_0:\mathbb{Q}_p]$. Then we take the Lubin-Tate's formal group 
$\mathcal{F}$ of $K$ such that $[\pi_K]=P(X)$, where $[-]:\mathcal{O}_K\isom \mathrm{End}(\mathcal{F})$.  We put $K_n$ is the Galois extension of $K$ 
generated by $[\pi_K^n]$-torsion points of $\mathcal{F}$ for any $n$, then we have a canonical isomorphism $\chi_{\mathrm{LT},n}:\mathrm{Gal}(K_n/K)\isom( \mathcal{O}_K^{\times}/
\pi^n\mathcal{O}_K)^{\times}$. We put $K_{\mathrm{LT}}:=\cup_{n=1}^{\infty}K_n$ and we put $G_n:=\mathrm{Gal}(K_n/K)$.

In [Ki08b] or [Ch09] etc, benign representation is defined as a special class of crystalline representations.
But, as we show in the sequel, we can easily generalize the main theorem to some potentially crystalline representations.
So, before defining benign representations, we first define the following class of potentially crystalline representations.

\begin{defn}
Let $W$ be an $E$-$B$pair. Then we say that $W$ is crystabelline if $W|_{G_{L}}$ is a crystalline $E$-$B$-pair of 
$G_{L}$ for a finite abel extension $L$ of $K$.
\end{defn}
\begin{rem}
Because a finite abel extension $L$ of $K$ is contained in $K_mL'$, where $L'$ is a finite unramified extension of $K$, by using Hilbert 90, we can easily show that $W$ is crystabelline if and only if 
$W|_{G_{K_m}}$ is crystalline for some $m\geqq 0$.
\end{rem}

Let $W$ be a crystabelline $E$-$B$-pair of rank $n$ such that $W|_{G_{K_m}}$ is crystalline for some $m$.
Let $D^{K_m}_{\mathrm{cris}}(W):=(B_{\mathrm{cris}}\otimes_{\mathbb{Q}_p}W)^{G_{K_m}}$ 
be the rank $n$ $E$-filtered ($\varphi,G_m$)-module over $K$. Then, because $K_m$ is totally ramified over $K$, this is 
a free $K_0\otimes_{\mathbb{Q}_p}E$-module of rank $n$. 
We take an embedding $\sigma:K_0\hookrightarrow \bar{E}$, this defines a map 
$\sigma:K_0\otimes_{\mathbb{Q}_p}E\rightarrow \bar{E}:x\otimes y\rightarrow \sigma(x)y$.
From this, we can define the $\sigma$-component  $D^{K_m}_{\mathrm{cris}}(W)_{\sigma}:=D^{K_m}_{\mathrm{cris}}(W)
\otimes_{K_0\otimes_{\mathbb{Q}_p}E, \sigma}\bar{E}$, this has an $\bar{E}$-linear 
$\varphi^f$-action and an $\bar{E}$-linear $G_m$-action. Let $\{\alpha_1,\cdots,\alpha_n\}$ be a solution in $\bar{E}$ 
(with multiplicities) of $\mathrm{det}_{\bar{E}}(T\cdot \mathrm{id}-\varphi ^f |_{D^{K_m}_{\mathrm{cris}}
(W)_{\sigma}})\in \bar{E}[T]$. Then, because $\varphi^f$ and $G_m$-action commute, so any generalized $\varphi^f$-eigenvector spaces 
 of $D^{K_m}_{\mathrm{cris}}(W)_{\sigma}$ are preserved by the action of $G_m$. 
 So we can take an $\bar{E}$-basis $e_{1,\sigma},\cdots,e_{n,\sigma}$ of $D^{K_m}_{\mathrm{cris}}(W)_{\sigma}$ such that, for any $i$,  $e_{i,\sigma}$ is a generalized eigenvector of 
 $\varphi^f$ with generalized eigenvalue $\alpha_i\in \bar{E}^{\times}$ and $G_m$ acts on $e_{i,\sigma}$ by a homomorphism $\tilde{\delta}_i:G_m\rightarrow \bar{E}^{\times}$.  We change numbering of $\{\alpha_1,\cdots,\alpha_n\}$ so that $e_{1,\sigma}, e_{2,\sigma},\cdots,e_{n,\sigma}$ gives a $\varphi^f$-Jordan decomposition of $D^{K_m}_{\mathrm{cris}}(W)_{\sigma}$ by this order. 
 Because $\{\sigma, \varphi^{-1}\sigma,\cdots, \varphi^{-(f-1)}\sigma\}=\mathrm{Hom}_{\mathbb{Q}_p}(K_0, \bar{E})$ and 
$\varphi^i:D^{K_m}_{\mathrm{cris}}(W)_{\sigma}\isom D^{K_m}_{\mathrm{cris}}(W)_{\varphi^{-i}\sigma}:x\otimes y\mapsto \varphi^i(x)\otimes y$ for any $x\in D^{K_m}_{\mathrm{cris}}(W)$ and $y\in \bar{E}$, is an $\bar{E}[\varphi^f,G_m]$-isomorphism, 
the set $\{\alpha_1,\cdots,\alpha_n\}$ doesn't depend on the choice of $\sigma:K_0\hookrightarrow \bar{E}$. 
Then, if we put $e_i:=e_{i,\sigma}+\varphi(e_{i,\sigma})+\cdots +\varphi^{f-1}(e_{i,\sigma})\in D^{K_m}_{\mathrm{cris}}(W)\otimes_{E}\bar{E}$, we have $D^{K_m}_{\mathrm{cris}}(W)\otimes_E\bar{E}=K_0\otimes_{\mathbb{Q}_p}\bar{E}e_1\oplus\cdots \oplus K_0\otimes_E\bar{E}e_n$ such that the subspace $K_0\otimes_{\mathbb{Q}_p}\bar{E}e_1\oplus\cdots\oplus K_0\otimes_{\mathbb{Q}_p}\bar{E}e_i$ is preserved by 
$\varphi$ and $G_m$-action for any $i$. Moreover, if we take a sufficiently large finite extension $E'$ of $E$, then we can
take $e_i\in D^{K_m}_{\mathrm{cris}}(W)\otimes_E E'$ such that $D^{K_m}_{\mathrm{cris}}(W)\otimes_E E'=
K_0\otimes_{\mathbb{Q}_p}E'e_1\oplus \cdots\oplus K_0\otimes_{\mathbb{Q}_p}E'e_n$ and $\alpha_i\in E'$ for any $i$ and $\tilde{\delta}_i:G_m\rightarrow E^{'\times}$.

By using these arguments, first we study a relation between crystabelline $E$-$B$-pairs and trianguline $E$-$B$-pairs.
\begin{lemma}\label{20.5}
Let $W$ be an $E$-$B$-pair of rank $n$. Then the following conditions are equivalent.
\begin{itemize}
\item[(1)]$W$ is crystabelline.
\item[(2)]$W$ is trianguline $($ i.e, there exists a finite extension $E'$ of $E$ such that $W\otimes_E E'$ is a
split trianguline $E'$-$B$-pair$)$ and potentially crystalline.
\end{itemize}
\end{lemma}
\begin{proof}
First we assume that $W$ is crystabelline. Then, by the above argument, for a sufficiently large finite extension 
$E'$ of $E$, we have $D^{K_m}_{\mathrm{cris}}(W)\otimes_E E'=K_0\otimes_{\mathbb{Q}_p}E'e_1\oplus 
\cdots \oplus K_0\otimes_{\mathbb{Q}_p}E'e_n$ as above. Then, for any $i$, $K_0\otimes_{\mathbb{Q}_p}E'e_1
\oplus \cdots\oplus K_0\otimes_{\mathbb{Q}_p}E'e_i$ is a sub $E'$-filtered ($\varphi,G_m$)-module 
of $D^{K_m}_{\mathrm{cris}}(W\otimes_EE')$. So, by Theorem $\ref{e}$, $W\otimes_E E'$ is 
split trianguline and potentially crystalline.

Next we assume that $W$ is trianguline and potentially crystalline. By extending coefficient, we may assume that $W$ is split trianguline. 
We take a triangulation $0\subseteq W_1\subseteq \cdots \subseteq W_n=W$ of $W$. 
Then, because sub objects and quotients of crystalline $B$-pairs are again crystalline, $W_i$ and $W_i/W_{i-1}$ are 
all potentially crystalline. Then, by Lemma 4.1 of [Na09], for sufficiently large $m$, $W_i/W_{i-1}|_{G_m}$ is crystalline. Then we claim that $W|_{G_m}$ is also crystalline. 
We prove this claim by induction on rank $n$ of $W$. When $n=1$, this is trivial. 
We assume that the claim is proved for $n-1$ case. Then $W_{n-1}|_{G_m}$ is crystalline.
If we put $W/W_{n-1}\isom W(\delta_n)$, we have $[W]\in \mathrm{H}^1(G_K, W_{n-1}(\delta_n^{-1}))$. By assumption, 
there exists a finite Galois extension $L$ of $K_m$ such that $[W]$ is contained in 
$\mathrm{Ker}(\mathrm{H}^1(G_K, W_{n-1}(\delta_n^{-1}))\rightarrow \mathrm{H}^1(G_L, B_{\mathrm{cris}}\otimes_{B_e}
(W_{n-1}(\delta_n^{-1}))_e))$.  Then  it suffices to prove that the natural map $\mathrm{H}^1(G_{K_m}, B_{\mathrm{cris}}\otimes_{B_e}(W_{n-1}(\delta_n^{-1}))_e)\rightarrow \mathrm{H}^1(G_L, B_{\mathrm{cris}}\otimes_{B_e}(W_{n-1}(\delta_n^{-1}))_e)$ is injective. By inflation 
 restriction sequence, the kernel of this map is $\mathrm{H}^1(\mathrm{Gal}(L/K_m), D^L_{\mathrm{cris}}(W_{n-1}(\delta_n^{-1})))=0$. 
 So $W|_{G_m}$ is crystalline, i.e. $W$ is crystabelline.

\end{proof}

From here, we consider a crystabelline $E$-$B$-pair $W$ such that 
$D^{K_m}_{\mathrm{cris}}(W)\isom K_0\otimes_{\mathbb{Q}_p}Ee_1\oplus \cdots \oplus 
K_0\otimes_{\mathbb{Q}_p}Ee_n$ such that 
$K_0\otimes_{\mathbb{Q}_p}Ee_i$ is preserved by ($\varphi,G_m$) and $\varphi^f(e_i)=\alpha_ie_i$ for 
some $\alpha_i\in E^{\times}$ such that $\alpha_i\not= \alpha_j$ for any $i\not= j$. 
Let $\mathfrak{S}_n$ be the $n$-th permutation group.
Then, for any $\tau\in \mathfrak{S}_n$, there is a filtration by $E$-filtered ($\varphi,G_m$)-modules on $D^{K_m}_{\mathrm{cris}}(W)$ ,
$\mathcal{F}_{\tau}:0\subseteq F_{\tau,1}\subseteq \cdots\subseteq F_{\tau, n-1}
\subseteq F_{\tau,n}=D^{K_m}_{\mathrm{cris}}(W)$  such that 
$F_{\tau,i}:=K_0\otimes_{\mathbb{Q}_p}Ee_{\tau(1)}\oplus \cdots\oplus K_0\otimes_{\mathbb{Q}_p}Ee_{\tau(i)}$ for any $1\leqq i \leqq n$
( the filtration on $F_{\tau,i}$ is induced from $D^{K_m}_{\mathrm{cris}}(W)$).
We put $\mathrm{gr}_{\tau,i}D^{K_m}_{\mathrm{cris}}(W):=F_{\tau,i}/F_{\tau,i-1}$ for any 
$1\leqq i\leqq n$. Then, by Theorem $\ref{e}$, there exists a 
filtration $\mathcal{T}_{\tau}:0\subseteq W_{\tau,1}\subseteq\cdots \subseteq W_{\tau,n-1}\subseteq W_{\tau,n}=W$ 
such that $W_{\tau,i}|_{G_m}$ is crystalline and $D^{K_m}_{\mathrm{cris}}(W_{\tau,i})=F_{\tau,i}$. 
Then $W_{\tau,i}/W_{\tau,i-1}$ is a rank one crystabelline $E$-$B$-pair such that 
$D^{K_m}_{\mathrm{cris}}(W_{\tau,i}/W_{\tau,i-1})\isom \mathrm{gr}_{\tau,i}D^{K_m}_{\mathrm{cris}}(W)$. So, by Lemma 4.1 of [Na09] and by its proof, 
there exists $\{k_{(\tau,i),\sigma}\}_{\sigma\in\mathcal{P}}\in\prod_{\sigma\in \mathcal{P}}\mathbb{Z}$ and 
homomorphisms $\tilde{\delta}_{i}:K^{\times}\rightarrow E^{\times}$  satisfying $\tilde{\delta}_{i}|_{1+\pi_K^m\mathcal{O}_K}=1$ and $\tilde{\delta}_{i}(\pi_K)=1$, 
such that 
$W_{\tau,i}/W_{\tau,i-1}\isom W(\delta_{\alpha_{\tau(i)}}\tilde{\delta}_{\tau(i)}\prod_{\sigma\in\mathcal{P}}\sigma(x)^{k_{(\tau,i),\sigma}})$ 
for any $1\leqq i\leqq n$, where $\delta_{\alpha_i}:K^{\times}\rightarrow E^{\times}$ is the homomorphism such that 
$\delta_{\alpha_i}|_{\mathcal{O}_K^{\times}}=1$ and $\delta_{\alpha_i}(\pi_K)=\alpha_i$. Then, for any $\sigma\in\mathcal{P}$, the set $\{k_{(\tau,1),\sigma},k_{(\tau,2),\sigma},\cdots,k_{(\tau,n),\sigma}\}$ is independent of $\tau\in \mathfrak{S}_n$  because these numbers are 
$\sigma$-part of Hodge-Tate weights of $W$. We write this set (with multiplicities) by $\{k_{1,\sigma},k_{2,\sigma},\cdots, k_{n,\sigma}\}$ 
such that $k_{1,\sigma}\geqq k_{2,\sigma}\geqq\cdots\geqq k_{n,\sigma}$ for any $\sigma\in\mathcal{P}$. Under these notations, we define the notion of benign $E$-$B$-pair as follows.
\begin{defn}\label{26}
Let $W$ be a rank $n$ crystabelline $E$-$B$-pair as above. Then 
we say that $W$ is a benign $E$-$B$-pair if the following conditions hold:
\begin{itemize}
\item[(1)] For any $i\not= j$, we have $\alpha_i/\alpha_j\not=1, p^f, p^{-f}$.
\item[(2)] For any $\sigma\in\mathcal{P}$, we have $0=k_{1,\sigma}>k_{2,\sigma}>\cdots >k_{n-1,\sigma}>k_{n,\sigma}$.
\item[(3)] For any $\tau\in \mathfrak{S}_n$ and $\sigma\in \mathcal{P}$, we have 
$k_{(\tau,i),\sigma}=k_{i,\sigma}$ for any $1\leqq i\leqq n$.
\end{itemize}
We say that $W$ is twisted benign if there exists a benign $E$-$B$-pair $W'$ and 
$\{k_{\sigma}\}_{\sigma\in\mathcal{P}}\in \prod_{\sigma\in \mathcal{P}}\mathbb{Z}$ 
such that $W\isom W'\otimes W(\prod_{\sigma\in\mathcal{P}}\sigma(x)^{k_{\sigma}})$.
(i.e. we drop the condition $k_{1,\sigma}=0$ in (2)).
\end{defn}
\begin{rem}
By definition, if $W$ is a twisted benign, then we have $W_{\tau,i}/W_{\tau,i-1}
\isom W(\delta_{\alpha_{\tau(i)}}\tilde{\delta}_{\tau(i)}\prod_{\sigma\in\mathcal{P}}\sigma(x)^{k_{i,\sigma}})$ 
for any $\tau\in\mathfrak{S}_n$ and $1\leqq i\leqq n$.
\end{rem}
\begin{rem}
When $n=2$ and $K=\mathbb{Q}_p$ and $W$ is crystalline, this definition is same as that of Kisin in [Ki08b].
When $n=2$ and $K=\mathbb{Q}_p$, the condition (3) is induced from the other weaker condition that $\mathrm{End}_{E[G_K]}(V)=E$.
For the tangent space theorem (so for the application to Zariski density), the author thinks that 
this condition (3) is the most important and subtle condition to verify.
But, in general case, i.e. $n\geqq 3$ and $K=\mathbb{Q}_p$ case, or $n\geqq 2$ and $K\not=\mathbb{Q}_p$ case, it seems difficult to induce 
the condition (3) from some other weaker conditions. In [Ch09], Chenevier calls the condition (3) to be 
 "generic".
\end{rem}
\begin{lemma}\label{21}
Let $W$ be a twisted benign $E$-$B$-pair.
If $W_1$ is a saturated sub $E$-$B$-pair of $W$, then 
$W_1$ and $W/W_1$ are also twisted benign $E$-$B$-pairs.
\end{lemma}
\begin{proof}
This follows from the definition and the fact that 
sub objects and quotients of crystabelline $E$-$B$-pairs 
are crystabelline $E$-$B$-pairs.
\end{proof}

\subsection{Deformations of benign $B$-pairs}

\begin{lemma}\label{22}
Let $W$ be a twisted benign $E$-$B$-pair of rank $n$. 
Then $W$ satisfies $\mathrm{End}_{G_K}(W)=E$ and 
for any $\tau\in \mathfrak{S}_n$, $\mathcal{T}_{\tau}:0\subseteq W_{\tau,1}\subseteq 
\cdots\subseteq W_{\tau,n-1}\subseteq W_{\tau,n}=W$ satisfies all the condition 
in Proposition $\ref{19}$. 
\end{lemma}
\begin{proof}
We prove $\mathrm{End}_{G_K}(W)=E$, other statements are trivial from 
(1) of  Definition $\ref{26}$. We prove this by induction on $n$, the rank of $W$. 
If  $n=1$, $\mathrm{End}_{G_K}(W)=\mathrm{H}^0(G_K, B_E)=E$.
We assume that the lemma is proved for rank $n-1$ case, let 
$W$ be a rank $n$ twisted benign $E$-$B$-pair. 
We take an element  $\tau\in\mathfrak{S}_n$ and consider the filtration 
$\mathcal{T}_{\tau}:0\subseteq W_{\tau,1}\subseteq\cdots \subseteq W_{\tau,n-1}\subseteq W_{\tau,n}=W$.
Then, by the above lemma, $W_{\tau,n-1}$ is twisted benign of rank $n-1$, so by induction hypothesis, 
we have $\mathrm{End}_{G_K}(W_{\tau,n-1})=E$.  
Let $f:W\rightarrow W$ be a non-zero morphism of $E$-$B$-pairs. Then, by definition of twisted benign $B$-pairs 
and by Proposition $\ref{i}$, we have $\mathrm{Hom}_{G_K}(W_{\tau,n-1}, W/W_{\tau,n-1})=0$. So $f(W_{\tau,n-1})
\subseteq W_{\tau,n-1}$. Because $\mathrm{End}_{G_K}(W_{\tau,n-1})=E$, we have $f|_{W_{\tau,n-1}}=a\cdot\mathrm{id}_{W_{\tau,n-1}}$ 
for some $a\in E$. If $a=0$, then $f:W\rightarrow W$ factors through a non-zero morphism $f':W/W_{\tau,n-1}\rightarrow W$. 
Because $\mathrm{Hom}_{G_K}(W/W_{\tau,n-1}, W_{\tau,n-1})=0$ by definition and by Proposition $\ref{i}$, 
the natural map $\mathrm{Hom}_{G_K}(W/W_{\tau,n-1},W)\hookrightarrow \mathrm{Hom}_{G_K}(W/W_{\tau,n-1}, W/W_{\tau,n-1})=E$
is injective, so the composition of $f'$ with the natural projection $W\rightarrow W/W_{\tau,n-1}$ induces  an isomorphism 
$W/W_{\tau,n-1}\isom W/W_{\tau,n-1}$. This implies that the short exact sequence 
$0\rightarrow W_{\tau,n-1}\rightarrow W\rightarrow W/W_{\tau,n-1}\rightarrow 0$ splits. 
Then we can choose a $\tau'\in\mathfrak{S}_n$, such that $W_{\tau',1}=W/W_{\tau,n-1}$ , and 
this $\tau'$ doesn't satisfy the condition (3) in the definition of benign $B$-pairs. It's contradiction. 
So the above $a$ must not be zero. In this case, consider the map $f-a\cdot\mathrm{id}_W\in \mathrm{End}_{G_K}(W)$. 
Then the same argument as above implies that $f=a\cdot\mathrm{id}_W$. So $\mathrm{End}_{G_K}(W)=E$.
\end{proof}
\begin{corollary}\label{23}
Let $W$ be a twisted benign $E$-$B$-pair of rank $n$. 
Then the functor $D_{W}$ is pro-representable by $R_W$ which is 
formally smooth of dimension $n^2[K:\mathbb{Q}_p]+1$. 
For any $\tau\in\mathfrak{S}_n$, the functor $D_{W,\mathcal{T}_{\tau}}$ is pro-representable 
by a quotient $R_{W,\mathcal{T}_{\tau}}$ of $R_W$ which is formally smooth 
of dimension $\frac{n(n+1)}{2}[K:\mathbb{Q}_p]+1$.
\end{corollary}
\begin{proof}
This follows from Proposition $\ref{19}$.
\end{proof}

Next, we want to consider the relation between $R_{W}$ and $R_{W,\mathcal{T}_{\tau}}$ for 
all $\tau\in \mathfrak{S}_n$. 
In particular, we want to compare the tangent space of $R_{W}$ and the sum of tangent spaces 
of $R_{W,\mathcal{T}_{\tau}}$ for all $\tau\in \mathfrak{S}_n$. 
For this, first we need to recall  the potentially crystalline  deformation functor.
\begin{defn}
Let $W$ be a potentially crystalline $E$-$B$-pair. 
Then we define potentially crystalline deformation functor $D_{W}^{\mathrm{cris}}$ which is a sub functor of $D_W$ 
defined by, for any $A\in \mathcal{C}_E$, $D_{W}^{\mathrm{cris}}(A):=\{[W_A]\in D_{W}(A)| W_A $ is potentially crystalline$\}$.
\end{defn}
\begin{lemma}\label{24}
Let $W$ be a potentially crystalline $E$-$B$-pair. 
If $\mathrm{End}_{G_K}(W)=E$, then $D_{W}^{\mathrm{cris}}$ is pro-representable by 
a quotient $R_{W}^{\mathrm{cris}}$ of $R_W$ which is formally smooth of dimension 
equal to $\mathrm{dim}_E(D_{\mathrm{dR}}(\mathrm{ad}(W))/\mathrm{Fil}^0D_{\mathrm{dR}}(\mathrm{ad}(W))) +
\mathrm{dim}_E(\mathrm{H}^0(G_K, \mathrm{ad}(W)))$. 
\end{lemma}
\begin{proof}
For pro-representability, by Proposition $\ref{12}$, it suffices to relatively representability of 
$D_{W}^{\mathrm{cris}}\hookrightarrow D_W$ as in the proof of Proposition $\ref{15}$. 
But in this case, the condition (1) and (2) is trivial and (3) follows from the fact 
that the sub object of potentially crystalline $E$-$B$-pair is again potentially crystalline. 
Formally smoothness follows from Proposition 3.1.2 and Lemma 3.2.1 of [Ki08a] by using deformations 
of filtered ($\varphi, G_K$)-modules. For dimension, we take a finite Galois extension $L$ of $K$ such that 
$W|_{G_L}$ is crystalline. Then, by the same argument as in the proof of Lemma $\ref{20.5}$, 
 any $W_A\in D^{\mathrm{cris}}_W(A)$ is crystalline when restricted to $G_L$. So it's easy to check 
that the map $D_{W}(E[\varepsilon])\isom \mathrm{H}^1(G_K, \mathrm{ad}(W))$ 
induces isomorphism $D_{W}^{\mathrm{cris}}(E[\varepsilon])\isom \mathrm{Ker}(\mathrm{H}^1(G_K, 
\mathrm{ad}(W))\rightarrow \mathrm{H}^1(G_L, B_{\mathrm{cris}}\otimes_{B_e} \mathrm{ad}(W)_e))$. Then again by the same way as in 
Lemma $\ref{20.5}$, the natural map $\mathrm{H}^1(G_K, B_{\mathrm{cris}}\otimes_{B_e}\mathrm{ad}(W)_e)\rightarrow \mathrm{H}^1(G_L, B_{\mathrm{cris}}\otimes_{B_e}\mathrm{ad}(W)_e)$ is 
injective. So we have isomorphism $D^{\mathrm{cris}}_W(E[\varepsilon])\isom \mathrm{Ker}(\mathrm{H}^1(G_K, \mathrm{ad}(W))
\rightarrow \mathrm{H}^1(G_K, B_{\mathrm{cris}}\otimes_{B_e}\mathrm{ad}(W)_e))$. We can calculate the dimension of this 
group in the same way as in the proof of  Proposition 2.7 [Na09].
\end{proof}
\begin{corollary}\label{25}
Let $W$ be a twisted benign $E$-$B$-pair of rank $n$. 
Then $R_{W}^{\mathrm{cris}}$ is formally smooth of dimension 
$\frac{(n-1)n}{2}[K:\mathbb{Q}_p]+1$.
\end{corollary}
\begin{proof}
This follows from Lemma $\ref{22}$ and $\mathrm{dim}_E(D_{\mathrm{dR}}(\mathrm{ad}(W))/\mathrm{Fil}^0D_{\mathrm{dR}}(\mathrm{ad}(W)))
=\frac{(n-1)n}{2}[K:\mathbb{Q}_p]$, which follows from the condition (2) in Definition $\ref{26}$.
\end{proof}

\begin{defn}\label{34}
Let $W$ be a twisted benign $E$-$B$-pair of rank $n$ such that $W|_{G_{K_m}}$ is crystalline. 
For any $\tau\in\mathfrak{S}_n$, we define a rank one saturated 
crystabelline $E$-$B$-pair $W'_{\tau}\subseteq W$ such that $D^{K_m}_{\mathrm{cris}}(W'_{\tau})=K_0\otimes_{\mathbb{Q}_p}E
e_{\tau(n)} \subseteq D^{K_m}_{\mathrm{cris}}(W)$. Then we define a sub functor 
$D^{\mathrm{cris}}_{W,\tau}$ of $D_W$ by $D^{\mathrm{cris}}_{W,\tau}(A):=\{[W_A]\in D_W(A)| $ there exists rank one crystabelline 
saturated sub $A$-$B$-pair $W'_A\subseteq W_A$ such that $W'_A\otimes_A E=W'_{\tau} \}$
\end{defn}

\begin{lemma}\label{27}
Under the above condition. The functor $D^{\mathrm{cris}}_{W,\tau}$ is pro-representable by 
a quotient $R^{\mathrm{cris}}_{W,\tau}$ of $R_W$ which is formally smooth and it's dimension 
is equal to $n(n-1)[K:\mathbb{Q}_p] +1$.
\end{lemma}
\begin{proof}
The proof of relatively representability of $D^{\mathrm{cris}}_{W,\tau}\hookrightarrow D_W$
 and the proof of formally smoothness are easily followed from combination of proofs  
 of Proposition $\ref{15}$ and Proposition $\ref{17}$ and Lemma $\ref{24}$.
 Here, we only prove dimension formula. 
 Let $\mathrm{ad}_{\tau}(W):=\{f\in\mathrm{ad}(W)| f(W'_{\tau})\subseteq 
 W'_{\tau}\}$. Then we have the following short exact sequence,
 \begin{equation*}
 0\rightarrow \mathrm{Hom}(W/W'_{\tau}, W)\rightarrow \mathrm{ad}_{\tau}(W)
 \rightarrow \mathrm{ad}(W'_{\tau})\rightarrow 0.
 \end{equation*}
 From this, by taking long exact sequence and by using Proposition $\ref{i}$, we have 
 the following short exact sequence,
 \begin{equation*}
 0\rightarrow \mathrm{H}^1(G_K, \mathrm{Hom}(W/W'_{\tau},W))\rightarrow 
 \mathrm{H}^1(G_K, \mathrm{ad}_{\tau}(W))\rightarrow \mathrm{H}^1(G_K, 
 \mathrm{ad}(W'_{\tau}))\rightarrow 0.
 \end{equation*}
 We define $\mathrm{H}^1_{f,\tau}(G_K, \mathrm{ad}_{\tau}(W))$ 
 by the inverse image of $\mathrm{H}^1_f(G_K,\mathrm{ad}(W'_{\tau}))$ 
 in $\mathrm{H}^1(G_K,\allowbreak \mathrm{ad}_{\tau}(W))$. So we have 
 a short exact sequence,
 \begin{equation*}
 0\rightarrow \mathrm{H}^1(G_K, \mathrm{Hom}(W/W'_{\tau},W))
 \rightarrow \mathrm{H}^1_{f,\tau}(G_K, \mathrm{ad}_{\tau}(W))
 \rightarrow \mathrm{H}^1_f(G_K, \mathrm{ad}(W'_{\tau}))\rightarrow 0.
 \end{equation*}
 Then, in the same way as in Lemma $\ref{24}$, we can show that the natural map $D_{W}(E[\varepsilon])\rightarrow \mathrm{H}^1(G_K, \mathrm{ad}(W))$ 
 induces an isomorphism $D^{\mathrm{cris}}_{W,\tau}(E[\varepsilon])\isom \mathrm{H}^1_{f,\tau}(G_K, 
 \allowbreak \mathrm{ad}_{\tau}(W))$. 
 By using Theorem $\ref{h}$ and Proposition $\ref{i}$, we can calculate that 
 $\mathrm{dim}_E\mathrm{H}^1(G_K, \mathrm{Hom}\allowbreak (W/W'_{\tau},W))=n(n-1)[K:\mathbb{Q}_p]$. 
 Because $\mathrm{ad}(W'_{\tau})=B_E$ is trivial, so we have $\mathrm{H}^1_{f}(G_K, \mathrm{ad}(W'_{\tau}))
 =1$ by Proposition 2.7 of [Na09]. So we have $\mathrm{dim}_E\mathrm{H}^1_{f,\tau}(G_K, \mathrm{ad}_{\tau}(W))\allowbreak =n(n-1)[K:\mathbb{Q}_p] +1$.
 This proves that $R_{W,\tau}^{\mathrm{cris}}$ is dimension $n(n-1)[K:\mathbb{Q}_p]+1$.

\end{proof}
\begin{lemma}\label{28}
Let $W$ be a twisted benign $E$-$B$-pair of rank $n$. 
Let $W_A$ be a deformation of $W$ over $A$ which is potentially crystalline, 
then $[W_A]\in D_{W,\mathcal{T}_{\tau}}(A)$ and $[W_A]\in D^{\mathrm{cris}}_{W,\tau}(A)$ for 
any $\tau\in \mathfrak{S}_n$.
\end{lemma}
\begin{proof}
Let $W_A$ be as above. If $W|_{G_{K_m}}$ is crystalline, 
then $W_A|_{G_{K_m}}$ is crystalline by the proof of Lemma $\ref{24}$. 
So it suffices to show that $D^{K_m}_{\mathrm{cris}}(W_A)\isom 
K_0\otimes_{\mathbb{Q}_p}Ae_1\oplus K_0\otimes_{\mathbb{Q}_p}Ae_2\oplus\cdots 
\oplus K_0\otimes_{\mathbb{Q}_p}Ae_n$ such that $K_0\otimes_{\mathbb{Q}_p}Ae_i$ is 
preserved by ($\varphi,G_{m}$) and $\varphi^f(e_i)=\tilde{\alpha_i}e_i$ for a lift $\tilde{\alpha_i}\in A^{\times}$ 
of $\alpha_i\in E^{\times}$ for any $1\leqq i\leqq n$.  For proving this claim, first we note that 
$D^{K_m}_{\mathrm{cris}}(W_A)$ is a free $K_0\otimes_{\mathbb{Q}_p}A$-module of rank $n$ and 
$D^{K_m}_{\mathrm{cris}}(W_A)\otimes_A E\isom D^{K_m}_{\mathrm{cris}}(W)$ by 
Proposition 1.3.4 and Proposition 1.3.5 [Ki09]. 
Then, for any $\sigma:K_0\hookrightarrow E$, we can write $D^{K_m}_{\mathrm{cris}}(W_A)_{\sigma}=
Ae_{1,\sigma}\oplus \cdots \oplus Ae_{n,\sigma}$ such that 
$\varphi^f(e_{i,\sigma})\equiv \alpha_ie_{i,\sigma}$ (mod $m_A$) for any $1\leqq i\leqq n$. 
Then, by easy linear algebra, we can take an $A$-basis $e'_{1,\sigma}, e'_{2,\sigma}, \cdots, e'_{n,\sigma}$ of 
$D^{K_m}_{\mathrm{cris}}(W_A)_{\sigma}$ such that $\varphi^f(e'_{i,\sigma})=\tilde{\alpha}_ie'_{i,\sigma}$ 
for a lift $\tilde{\alpha}_i\in A^{\times}$ of $\alpha_i$ for any $i$. Because $\varphi$ and $G_{m}$ commute and $\alpha_i\not=\alpha_j$, 
so $Ae'_{i,\sigma}$ is preserved by $G_{m}$. If we take $e_i:=e'_{i,\sigma}+\varphi(e'_{i,\sigma})+\cdots 
\varphi^{f-1}(e'_{i,\sigma})\in D^{K_m}_{\mathrm{cris}}(W_A)$, then we have $D^{K_m}_{\mathrm{cris}}(W_A)=
K_0\otimes_{\mathbb{Q}_p}Ae_1\oplus \cdots \oplus K_0\otimes_{\mathbb{Q}_p}Ae_n$ satisfying 
the property of the claim.

\end{proof}

\begin{lemma}\label{29}
Let $W$ be a twisted benign $E$-$B$-pair of rank $n$ and $\tau\in \mathfrak{S}_n$ such that 
$W|_{G_{K_m}}$ is crystalline.
Let $[W_A]\in D_{W,\mathcal{T}_{\tau}}(A)$ be a trianguline deformation over $A$ 
with a lifting of  triangulation $0\subseteq W_{1,A}\subseteq W_{2,A}\subseteq \cdots W_{n,A}=W_A$. 
If $W_{i,A}/W_{i-1,A}$ is Hodge-Tate for any $1\leqq i\leqq n$, then $W_A|_{G_{K_m}}$ is crystalline.
\end{lemma}
\begin{proof}
First, we prove that $(W_{i,A}/W_{i-1,A})|_{G_{K_m}}$ is crystalline with Hodge-Tate weight $\{k_{i,\sigma}\}_{\sigma\in \mathcal{P}}$. 
Because $W_{i,A}/W_{i-1,A}$ is written as successive extensions of $W_{\tau,i}/W_{\tau,i-1}$, $W_{i,A}/W_{i-1,A}$
has Hodge-Tate weight $\{k_{i,\sigma}\}_{\sigma\in\mathcal{P}}$. Then, twisting $W_A$ by a crystalline character $\delta_{\alpha_{\tau(i)}}^{-1}
\prod_{\sigma\in\mathcal{P}}\sigma(x)^{-k_{i,\sigma}}:K^{\times}\rightarrow A^{\times}$, we may assume that 
$W_{i,A}/W_{i-1, A}$ is an $\mathrm{\acute{e}}$tale Hodge-Tate $A$-$B$-pair of rank one with Hodge-Tate weight zero and is 
a deformation of $\mathrm{\acute{e}}$tale potentially unramified $E$-$B$-pair $W(\tilde{\delta}_{\tau(i)})$. Then, by Sen's theorem 
([Se73] or Proposition 5.24 of [Be02]),
$W_{i,A}/W_{i-1,A}$ is potentially unramified, so there exists $\delta:K^{\times}\rightarrow A^{\times}$ a unitary 
homomorphism such that $\delta|_{\mathcal{O}_K^{\times}}$ has a finite image and $W_{i,A}/W_{i-1,A}\isom W(\delta)$ and 
$\delta$ is a lift of $\tilde{\delta}_{\tau(i)}$. Then, because $(1+m_A)\cap A^{\times}_{tor}=\{1\}$, 
$\delta|_{\mathcal{O}_K^{\times}}=\tilde{\delta}_{\tau(i)}|_{\mathcal{O}_K^{\times}}:\mathcal{O}_K^{\times}\rightarrow A^{\times}$, so $W_{i,A}/W_{i-1,A}|_{G_{K_m}}$ is crystalline.
Next, we prove that $W_A|_{G_{K_m}}$ is crystalline by induction on rank $W$. 
When $n=1$, we just have proved this. 
Assume that the lemma is proved for rank $n-1$ case. 
Then $W_{n-1, A}|_{G_{K_m}}$ is crystalline. If we put $W_A/W_{n-1,A}\isom W(\delta_{A,n})$, then we have 
$[W_A]\in \mathrm{H}^1(G_K, W_{n-1,A}(\delta_{A,n}^{-1}))$. By considering Hodge-Tate weight of 
$W_{A}$ and condition (3) of Definition $\ref{26}$, we get $\mathrm{Fil}^0D_{\mathrm{dR}}(W_{n-1,A}(\delta_{A,n}^{-1}))=0$. 
By comparing dimensions by using Proposition 2.7 of [Na09], we have $\mathrm{H}^1_f(G_K,W_{A,n-1}(\delta_{A,n}^{-1}))=\mathrm{H}^1(G_K, 
W_{A,n-1}(\delta_{A,n}^{-1}))$. So  $W_A|_{G_{K_m}}$ is crystalline.

\end{proof}

\begin{lemma}\label{30}
Let $W$ be a twisted benign $E$-$B$-pair of rank $n$ and $W_1$ be a rank one crystabelline sub $E$-$B$-pair of $W$. 
Then the saturation $W_{1}^{sat}:=(W^{sat}_{1,e}, W^{+, sat}_{1,\mathrm{dR}})$ of $W_1$ in $W$ is 
crystabelline and the natural map $\mathrm{Hom}_{G_K}( W_{1}^{sat}, W)
\rightarrow \mathrm{Hom}_{G_K}(W_1,W)$ induced by natural 
inclusion $W_1\hookrightarrow W_{1}^{sat}$ is isomorphism 
between one dimensional $E$-vector spaces.
\end{lemma}
\begin{proof}
Because $W_{1,e}=W^{sat}_{1,e}$ by Lemma 1.14 of [Na09], so $W_{1}^{sat}$ is crystabelline.
By definition of twisted benign $E$-$B$-pairs, the Hodge-Tate weight of $W_{1}^{sat}$ 
is $\{k_{1,\sigma}\}_{\sigma\in\mathcal{P}}$. Then consider the following short exact sequence 
of complexes of $G_K$-modules defined as in p.890 of [Na09], 
\begin{equation*}
0\rightarrow \mathrm{C}^{\bullet}(W\otimes (W_{1}^{sat})^{\vee})\rightarrow \mathrm{C}^{\bullet}(W
\otimes W_{1}^{\vee})\rightarrow  
(W\otimes W_{1}^{\vee})^+_{\mathrm{dR}})/
(W\otimes(W_{1}^{sat})^{\vee})^+_{\mathrm{dR}}) [0]\rightarrow 0.
\end{equation*}
From this, we have the exact sequence,
\[
\begin{array}{ll}
0\rightarrow \mathrm{H}^0(G_K, W\otimes (W_1^{sat})^{\vee})&\rightarrow 
\mathrm{H}^0(G_K, W\otimes W_1^{\vee}) \\
 &\rightarrow \mathrm{H}^0(G_K, (W\otimes W_1^{\vee})^+_{\mathrm{dR}}/
(W\otimes(W_1^{sat})^{\vee})^+_{\mathrm{dR}})\rightarrow \cdots
\end{array}
\]
By the condition (3) in Definition $\ref{26}$, we have $\mathrm{dim}_E\mathrm{H}^0(G_K, (W\otimes (W_1^{sat})^{\vee})^+_{\mathrm{dR}})
=\mathrm{dim}_E\mathrm{H}^0(G_K, (W\otimes W_1^{\vee})^+_{\mathrm{dR}})=n[K:\mathbb{Q}_p]$, So, by a usual argument 
of de Rham representations, we have $\mathrm{H}^0(G_K, (W\otimes W_1^{\vee})^+_{\mathrm{dR}}/(W\otimes (W_1^{sat})^{\vee})^+_{\mathrm{dR}})=0$. So the map $\mathrm{H}^0(G_K, W\otimes (W_1^{sat})^{\vee}))\isom \mathrm{H}^0(G_K, W\otimes (W_1)^{\vee})
$ is isomorphism. Finally, for dimension, by the condition (1) in Definition $\ref{26}$ and Proposition $\ref{i}$ of [Na09], 
we have $\mathrm{dim}_E\mathrm{H}^0(G_K, W\otimes (W_1^{sat})^{\vee})=1$. These prove the lemma.

\end{proof}
\begin{lemma}\label{31}
Let $W$ be a twisted benign $E$-$B$-pair and let $W_A$ be a deformation 
of $W$ over $A$. If there exists a rank one crystabelline sub $A$-$B$-pair 
$W_{1,A}\subseteq W_A$  such 
that $W_1:=W_{1,A}\otimes_A E\hookrightarrow W_A\otimes_{A}E$ remains injective.
Then the saturation $W_{1,A}^{sat}$ of $W_{1,A}$ in $W_A$ as an $E$-$B$-pair is 
a crystabelline $A$-$B$-pair and $W_A/W^{sat}_{1,A}$ is an $A$-$B$-pair and 
$W_{1,A}^{sat}\otimes_A E\isom W_1^{sat} (\subseteq W)$.
\end{lemma}
\begin{proof}
First, by Proposition 2.14 of [Na09], there exists $\{l_{\sigma}\}_{\sigma\in\mathcal{P}}
\in \prod_{\sigma\in\mathcal{P}}\mathbb{Z}_{\leqq 0}$ such that  $W_1^{sat}\isom W_1\otimes W(\prod_{\sigma\in\mathcal{P}}
\sigma(x)^{l_{\sigma}})$. Then we claim that the inclusion 
$\mathrm{Hom}_{G_K}(W_{1,A}\otimes W_A(\prod_{\sigma\in\mathcal{P}}\sigma(x)^{l_{\sigma}}), W_A)
\rightarrow \mathrm{Hom}_{G_K}(W_{1,A}, W_A)$ induced by the natural inclusion $W_{1,A}\hookrightarrow W_{1,A}
\otimes W_A(\prod_{\sigma\in\mathcal{P}}\sigma(x)^{l_{\sigma}})$ is isomorphism and 
that these groups are rank one free $A$-modules. By the same argument as in the last lemma, 
the cokernel of $\mathrm{Hom}_{G_K}(W_{1,A}\otimes W_A(\prod_{\sigma\in\mathcal{P}}\sigma(x)^{l_{\sigma}}), W_A)
\rightarrow \mathrm{Hom}_{G_K}(W_{1,A}, W_A)$ is contained in $\mathrm{H}^0(G_K, (W_A\otimes W_{1,A}^{\vee})^+_{\mathrm{dR}}/
(W_A\otimes W_{1,A}^{\vee}\otimes W_A(\prod_{\sigma\in \mathcal{P}}\sigma(x)^{-l_{\sigma}}))^+_{\mathrm{dR}})$. This is zero 
from the proof of the last lemma. So the natural inclusion $\mathrm{Hom}_{G_K}(W_{1,A}\otimes W_A(\prod_{\sigma\in\mathcal{P}}\sigma(x)^{l_{\sigma}}), W_A)
\rightarrow \mathrm{Hom}_{G_K}(W_{1,A}, W_A)$ is isomorphism. Next, we prove, by induction on length of A, 
that $\mathrm{Hom}_{G_K}(W_{1,A}, W_A)$ is a free $A$-module of rank one.
When $A=E$, this claim is proved in the above Lemma $\ref{30}$. 
Assume that $A$ is length $n$ and, for a small extension $A\rightarrow A'$, assume that the claim is 
proved for $W_{A'}:=W\otimes_A A'$. Put $I$ the kernel of $A\rightarrow A'$, $W_{1,A'}:=W_{1,A}\otimes_A A'$. From the following  exact sequence,
\begin{equation*}
0\rightarrow I\otimes_E \mathrm{Hom}_{G_K}(W_1, W)
\rightarrow \mathrm{Hom}_{G_K}(W_{1,A}, W_A)\rightarrow \mathrm{Hom}_{G_K}(W_{1,A'}, W_{A'})
\end{equation*}
 and induction hypothesis, we have 
$\mathrm{length}\mathrm{Hom}_{G_K}(W_{1,A}, W_A)\leqq 
\mathrm{length}A$. On the other hand, the fact that given inclusion $\iota:W_{1,A}\hookrightarrow W_A$ remains 
injective after tensoring $E$ and the fact that $\mathrm{dim}_E\mathrm{Hom}_{G_K}(W_1, W)=1$ and induction hypothesis 
imply that the map $A\rightarrow \mathrm{Hom}_{G_K}(W_{1,A}, W_A): a\mapsto a\cdot\iota $ is injective. So we have 
$\mathrm{length}(A)= \mathrm{length} \mathrm{Hom}_{G_K}(W_{1,A}, W_A)$. These prove the claim for $A$.
From these  claims, the given inclusion $\iota:W_{1,A}\hookrightarrow W_A$  factors through a map $\tilde{\iota}:W_{1,A}':=W_{1, A}\otimes W_A(\prod_{\sigma}\sigma(x)^{l_{\sigma}})
\rightarrow W_A$ and this map is injective because the injection of morphisms of $B$-pair determined from the injection 
of $W_e$-part of $B$-pairs and $W_{1,A,e}=(W_{1,A}\otimes W_A(\prod_{\sigma\in\mathcal{P}}\sigma(x)^{l_{\sigma}}))_e$. Then we claim that this $\tilde{\iota}$ gives $W_{1,A}^{sat}\isom W_{1,A}'$ and claim that this satisfies all the properties in this lemma. By induction on length $A$, 
we assume that this claim is proved for $A'$.  First, we prove $W_{A}/W_{1,A}'$ is an $E$-$B$-pair. For proving this, by Lemma 2.1.4 of [Ber08], 
it suffices to show that $W_{A,\mathrm{dR}}^{+}/W_{1,A,\mathrm{dR}}^{'+}$ is a free $B_{\mathrm{dR}}^+$-module. By snake lemma, we have 
\begin{equation*}
0\rightarrow I\otimes_E W^+_{\mathrm{dR}}/W^{sat +}_{1,\mathrm{dR}}\rightarrow 
W^+_{A,\mathrm{dR}}/W^{' +}_{1, A,\mathrm{dR}}\rightarrow W^+_{A',\mathrm{dR}}/(W_{1,A,\mathrm{dR}}^{' +}\otimes_A A')\rightarrow 0.
\end{equation*}
From this and induction hypothesis, $W^+_{A,\mathrm{dR}}/W^{' +}_{1,A,\mathrm{dR}}$ is a free $B^+_{\mathrm{dR}}$-module.
Finally, we prove $A$-flatness of $W_{A}/W_{1,A}^{'}$. This follows from the fact that  the map 
$\tilde{\iota}\otimes \mathrm{id}_E:W_{1,A}^{'}\otimes_A E\hookrightarrow W_{A}\otimes_A E$ is saturated (we can see this 
from the proof of the above first claim). So $W_A/W_{1,A}^{'}$ is an $A$-$B$-pair. We finish the proof of this lemma.

\end{proof}

\begin{lemma}\label{32}
Let $W$ be a twisted benign $E$-$B$-pair of rank $n$.
For any $\tau\in\mathfrak{S}_n$, we have $D_{W,\mathcal{T}_{\tau}}(E[\epsilon])
+D^{\mathrm{cris}}_{W,\tau}(E[\varepsilon])=D_{W}(E[\varepsilon])$.
\end{lemma}
\begin{proof}
By Corollary $\ref{23}$ and  Lemma $\ref{24}$ and Lemma $\ref{27}$, we have $\mathrm{dim}_ED_{W}(E[\varepsilon])\allowbreak +\mathrm{dim}_ED_W^{\mathrm{cris}}(E[\varepsilon])
=\mathrm{dim}_E D_{W,\mathcal{T}_{\tau}}(E[\varepsilon])+\mathrm{dim}_ED^{\mathrm{cris}}_{W,\tau}(E[\varepsilon])$. 
So it suffices to show that $D_{W,\mathcal{T}_{\tau}}(A)\cap D^{\mathrm{cris}}_{W,\tau}(A)=D_{W}^{\mathrm{cris}}(A)$ for 
any $A\in \mathcal{C}_E$. First, the fact that $D^{\mathrm{cris}}_W(A)\subseteq D_{W,\mathcal{T}_{\tau}}(A)\cap D_{W,\tau}(A)$ 
follows from Lemma $\ref{28}$. 

We prove that $D_{W,\mathcal{T}_{\tau}}(A)\cap D^{\mathrm{cris}}_{W,\tau}(A)\subseteq D_{W}^{\mathrm{cris}}(A)$ by induction on
rank $n$ of $W$. When  $n=1$, this is trivial.
Let $W$ be rank $n$ and assume that the lemma is proved for rank  $n-1$.
Let $[W_A]\in D_{W,\mathcal{T}_{\tau}}(A)\cap D^{\mathrm{cris}}_{W,\tau}(A)$.
Then, by definition of $D_{W,\mathcal{T}_{\tau}}$ and $D^{\mathrm{cris}}_{W,\tau}$, there exists an $A$-triangulation 
$0\subseteq W_{1,A}\subseteq W_{2,A}\subseteq \cdots \subseteq W_{n-1,A}\subseteq W_{n,A}=W_A$ 
such that $W_{i,A}\otimes_A E\isom W_{\tau,i}$ for any $i$ and there exists saturated crystabelline rank one $A$-$B$-pair 
$W'_{1,A}\subseteq W_{A}$ such that $W'_{1,A}\otimes_A E\isom W'_{\tau}$. 
First, we claim that the composition of $W'_{1,A}\hookrightarrow W_A$ and $W_A\rightarrow W_A/W_{1,A}$ 
is  injective. By snake lemma, $\mathrm{Ker}(W'_{1,A}\rightarrow W_A/W_{1,A})$ is a sub $E$-$B$-pair of 
$W_{1, A}$. By the condition (1) of  Definition $\ref{26}$ and by Proposition 2.14 of [Na09], we have 
$\mathrm{Hom}_{G_K}(\mathrm{Ker}(W'_{1, A}\rightarrow W_A/W_{1,A}), W_{1,A})=0$, so $W'_{1,A}\rightarrow 
W_A/W_{1,A}$ is injective. Then, by Lemma $\ref{31}$, the saturation $(W'_{1,A})^{sat}$ of $W'_{1,A}$ in $W_A/W_{1,A}$ 
is a rank one crystabelline $A$-$B$-pair satisfying the same conditions as those of $W'_{1,A}\hookrightarrow W_A$. 
So, by induction hypothesis, $W_A/W_{1,A}$ is crystabelline. Then, by the condition (3) of Definition $\ref{26}$, Hodge-Tate weight of 
$W_A/W_{1,A}$ is $\{k_{2,\sigma},k_{3,\sigma},\cdots,k_{n,\sigma}\}_{\sigma}$ (with multiplicity $[A:E]$). 
Moreover, by (3) of Definition $\ref{26}$, crystabelline $W'_{1,A}$ has Hodge-Tate weight $\{k_{1,\sigma}\}_{\sigma\in\mathcal{P}}$ 
(with multiplicity). Because $k_{1,\sigma}\not= k_{i,\sigma}$ for any $i\not=1$ and there is no extension between different 
Hodge-Tate weight objects by a theorem of Tate, these imply that $W_A$ is a Hodge-Tate $E$-$B$-pair. So, 
by Lemma $\ref{29}$, $W_A$ is crystabelline. So we have that $[W_A]\in D_{W}^{\mathrm{cris}}(A)$.

\end{proof}
The following is the main theorem of this article, which is a crucial theorem for applications to Zariski density.
This theorem was discovered by Chenevier ([Che08], [Ch09]) for any rank case when $K=\mathbb{Q}_p$. 
\begin{defn}
For $R=R_W$ or $R_{W,\mathcal{T}_{\tau}}$, we put $t(R):=\mathrm{Hom}_E(m_R/m^2_R,E)$ the tangent 
space of $R$. Then, for any $\tau\in \mathfrak{S}_n$, we have a canonical inclusion $t(R_{W,\mathcal{T}_{\tau}})\hookrightarrow 
t(R_W)$.
\end{defn}

\begin{thm}\label{33}
Let $W$ be a twisted benign $E$-$B$-pair of rank $n$. 
Then we have $$\sum_{\tau\in\mathfrak{S}_n}t(R_{W,\mathcal{T}_{\tau}})=t(R_W).$$
\end{thm}
\begin{proof}
We prove this by induction on $n$. 
When $n=1$, then the theorem is trivial.
Assume that the theorem is true for rank $n-1$ case. 
Let $W$ be a rank $n$ twisted benign $E$-$B$-pair.
We take an element $\tau\in\mathfrak{S}_n$. 
Then we consider the sub functor $D_{W,\tau}$ of $D_W$ by 
$D_{W,\tau}(A):=\{[W_A]\in D_{W}(A)| $ there exists a rank one sub $A$-$B$-pair $W_{1,A}\subseteq W_A$ such 
that the quotient $W_A/W_{1,A}$ is an $A$-$B$-pair and $W_{1,A}\otimes_A E\isom W'_{\tau}\}$ ($W'_{\tau}$ is defined in 
Definition $\ref{34}$).
 Then $D^{\mathrm{cris}}_{W,\tau}$ is a sub functor of $D_{W,\tau}$ 
and we can show that $D_{W,\tau}(E[\varepsilon])\isom \mathrm{H}^1(G_K, \mathrm{ad}_{\tau}(W))$, here $\mathrm{ad}_{\tau}(W)
:=\{f \in \mathrm{ad}(W)| f(W'_{\tau})\subseteq W'_{\tau} \}$. So, by Lemma $\ref{32}$, we have $\mathrm{H}^1(G_K, \mathrm{ad}_{\tau}(W))+
\mathrm{H}^1(G_K, \mathrm{ad}_{\mathcal{T}_{\tau}}(W))=\mathrm{H}^1(G_K, \mathrm{ad}(W))$. 
Because, for any $\tau'\in\mathcal{S}_{\tau,n}:=\{\tau'\in\mathfrak{S}_n| \tau'(1)=\tau(n)\}$,  we have 
$D_{W,\mathcal{T}_{\tau'}}\subseteq D_{W,\tau}$. So we have natural maps $\mathrm{H}^1(G_K, \mathrm{ad}_{\mathcal{T}_{\tau'}}(W))
\rightarrow \mathrm{H}^1(G_K, \mathrm{ad}_{\tau}(W))$. Then it suffices to prove that 
the map $\oplus_{\tau'\in\mathcal{S}_{\tau,n}}\mathrm{H}^1(G_K, \mathrm{ad}_{\mathcal{T}_{\tau'}}(W))\rightarrow 
\mathrm{H}^1(G_K, \mathrm{ad}_{\tau}(W))$ is surjective. We prove this surjection as follows. 
First, by definition, we have the following short exact sequences of $E$-$B$-pairs for any $\tau'\in\mathcal{S}_{\tau,n}$,
\begin{equation}
0\rightarrow \mathrm{Hom}(W/W'_{\tau}, W)\rightarrow \mathrm{ad}_{\tau}(W)\rightarrow \mathrm{ad}(W'_{\tau})
\rightarrow 0,
\end{equation}
\begin{equation}
0\rightarrow \{f\in\mathrm{ad}_{\mathcal{T}_{\tau'}}(W)| f(W'_{\tau})=0\}\rightarrow \mathrm{ad}_{\mathcal{T}_{\tau'}}(W)\rightarrow 
\mathrm{ad}(W'_{\tau})\rightarrow 0.
\end{equation}
Moreover, we have
\begin{equation}
0\rightarrow \mathrm{Hom}(W/W'_{\tau}, W'_{\tau})\rightarrow \mathrm{Hom}(W/W'_{\tau}, W)\rightarrow \mathrm{ad}(W/W'_{\tau})\rightarrow 0,
\end{equation}
\begin{equation}
0\rightarrow \mathrm{Hom}(W/W'_{\tau}, W'_{\tau})\rightarrow \{f\in\mathrm{ad}_{\mathcal{T}_{\tau'}}(W)|f(W'_{\tau})=0\}
\rightarrow \mathrm{ad}_{\mathcal{T}_{\bar{\tau}'}}(W/W'_{\tau})\rightarrow 0.
\end{equation}
Here, for any $\tau'\in\mathcal{S}_{\tau,n}$, we put $\mathcal{T}_{\bar{\tau}'}:0\subseteq W_{\tau',2}/W'_{\tau}\subseteq W_{\tau',3}/W'_{\tau}
\subseteq \cdots \subseteq W_{\tau',n-1}/W'_{\tau}\subseteq W/W'_{\tau}$ the induced triangulation from $\mathcal{T}_{\tau'}$ 
on $W/W'_{\tau}$. We have $\mathrm{H}^2(G_K, \mathrm{Hom}(W/W'_{\tau}, W'_{\tau}))=0$ by the condition $(1)$ of Definition $\ref{26}$ and 
Proposition $\ref{i}$. We have  $\mathrm{H}^2(G_K, \mathrm{ad}_{\mathcal{T}_{\tau}}(W))=0$ from the proof of Proposition 
$\ref{19}$. So from the short exact sequence (4) above, we have $\mathrm{H}^2(G_K, \{f\in\mathrm{ad}_{\tau'}(W)|f(W'_{\tau})=0\})=0$.
From this and from (1) and (2) above, for proving the surjection of $\oplus_{\tau'\in \mathcal{S}_{\tau,n}}\mathrm{H}^1(G_K,\mathrm{ad}_{\mathcal{T}_{\tau'}}(W))\allowbreak 
\rightarrow \mathrm{H}^1(G_K, \mathrm{ad}_{\tau}(W))$, it suffices to prove 
the surjection of $\oplus_{\tau'\in \mathcal{S}_{\tau,n}}\mathrm{H}^1(G_K, \{f\in\mathrm{ad}_{\tau'}(W)|f(W'_{\tau})=0\})
\rightarrow \mathrm{H}^1(G_K, \mathrm{Hom}(W/W'_{\tau}, W))$. By (3) and (4) above and by $\mathrm{H}^2(G_K, \mathrm{Hom}(W/W'_{\tau}, W'_{\tau}))=0$, 
this surjection follows from the surjection of $\oplus_{\tau'\in\mathcal{S}_{\tau,n}}\mathrm{H}^1(G_K, \mathrm{ad}_{\mathcal{T}_{\bar{\tau}'}}(W/W'_{\tau}))\rightarrow \mathrm{H}^1(G_K, \mathrm{ad}(W/W'_{\tau}))$, which  is the induction hypothesis. So we have finished the proof of this theorem

\end{proof}

\end{document}